\documentclass[11pt, reqno]{amsart} 
\usepackage{epstopdf}
\usepackage{amssymb}
\usepackage{textcomp}
\usepackage{amsmath, amsthm, amscd, amsfonts}
\usepackage{graphicx, float}
\usepackage{mathrsfs}
\usepackage{bbm}

\usepackage[left= 3 cm, right= 3 cm, top= 2.6 cm, bottom=2.8 cm, foot= 1 cm]{geometry}

\usepackage{hyperref}
\hypersetup{colorlinks=true, citecolor=blue}
\usepackage{color}
\usepackage[usenames,dvipsnames,svgnames,table]{xcolor}

\newcommand{\mb}{\mathbb}

\newcommand{\mf}{\mathfrak}
\newcommand{\mc}{\mathcal}

\newcommand{\ov}{\overline}
\newcommand{\ora}{\overrightarrow}
\newcommand{\wt}{\widetilde}

\makeatletter 
\@addtoreset{equation}{section}
\makeatother  

\numberwithin{equation}{section}

\title{$U(1)$-vortices and quantum Kirwan map}

\author[Xu]{Guangbo Xu}
\address{
Department of Mathematics \\
Princeton University \\
Fine Hall, Washington Road \\
Princeton, NJ 08544 USA \\
312-646-9515
}
\email{guangbox@math.princeton.edu}

\begin{document}

	\newtheorem{thm}{Theorem}[section]
	\newtheorem{lemma}[thm]{Lemma}
	\newtheorem{cor}[thm]{Corollary}
	\newtheorem{prop}[thm]{Proposition}
	
	\theoremstyle{definition}
	\newtheorem{defn}[thm]{Definition}
	
	\theoremstyle{remark}
	\newtheorem{rem}[thm]{Remark}
	\newtheorem{hyp}[thm]{Hypothesis}
	\newtheorem{example}[thm]{Example}

\setcounter{tocdepth}{1}

\maketitle

\begin{abstract}
We study the symplectic vortex equation over the complex plane, for the target space ${\mb C}^N$ ($N\geq 2$) with diagonal $U(1)$-action. We classify all solutions with finite energy and identify their moduli spaces, which generalizes Taubes' result for $N=1$. We also studied their compactifications and use them to compute the associated quantum Kirwan maps $\kappa_Q: H^*_{U(1)}({\mb C}^N) \to QH^*({\mb P}^{N-1})$.
\end{abstract}

\tableofcontents

\section{Introduction}

This note is devoted to the understanding of the geometry of the vortex equation over the complex plane and their moduli spaces. The vortex equation over ${\mb C}$ and the moduli spaces of solutions are the central objects in the project of Ziltener (cf. \cite{quantumkirwan}) to define a quantum version of the Kirwan map for symplectic manifold $X$ with Hamiltonian action by a compact Lie group $G$. Here we work with a concrete example, where, the target manifold is the vector space ${\mb C}^N$ acted by $U(1)$ via complex multiplication and the symplectic quotient of this action is the projective space ${\mb P}^{N-1}$. 

Many years ago, in \cite{Taubes_vortex}, Taubes gave the classification of finite energy vortices in the case where the target $X= {\mb C}$ with $U(1)$-action. The moduli space for each ``vortex number'' $d>0$ is ${\rm Sym}^d {\mb C}$, the $d$-fold symmetric product of the complex plane. For target ${\mb C}^N$ with $N\geq 2$ and the diagonal $U(1)$-action there was no result of either construction of nontrivial solutions or classification, until very recently Venugopalan-Woodward (\cite{VW_Classification}) claim that, for target manifold a projective variety acted by a reductive Lie group (including our case) using heat flow method one can identify the solutions to the vortex equation and algebraic maps from ${\mb C}$ to the ``quotient stack'' $X/G$ with certain condition on the assymptotic behavior at infinity.

For the special example considered in this note, we take a different approach which can be certainly extended to more general situations. This approach contains three basic ingredients here. The first one is the adiabatic limit analysis for vortex equation over a compact Riemann surface $\Sigma$ with growing area form, most of which is provided by \cite{Gaio_Salamon_2005}. In particular, it implies that (in the symplectic aspherical case) as we grow the area form on $\Sigma$, the energy density of solutions blows up at most in the same rate as we enlarging the surface. The second ingredient is the Hitchin-Kobayashi correspondence for stable $N$-pairs over the compact $\Sigma$ provided by Bradlow's theorem \cite{Bradlow_stable_pairs} (see Theorem \ref{thm21}). This correspondence works for any large area form on the domain curve, hence we can identify moduli spaces of solutions for different area forms with the same algebraic moduli space. The third ingredient is the observation that the energy concentration of a sequence of solutions has an algebraic description (in our example this means a sequence of holomorphic $N$-pairs develops a base point). With the understanding of these ingredients, we can manipulate the bubbling in the algebraic moduli space and construct all possible solutions. 

It is worth pointing out that our approach is very elementary, and by looking at concrete examples, it helps understand the behavior of the vortices (for example, why they merge to the moment level surface and converge to holomorphic spheres in the symplectic quotient). One can also generalize this method to other cases, for example, toric manifolds and flag manifolds, both as symplectic quotients of Euclidean spaces.

A purpose of classifying affine vortices and identifying their moduli spaces is to compute the quantum version of the Kirwan map, which was proposed by D. Salamon and a rigorous definition relies on an ongoing project of Ziltener. In the case we considered in this note, we can identify the moduli space and a natural compactification (we call it the Uhlenbeck compactification), which we can use to compute the quantum Kirwan map $\kappa_Q: H_{U(1)}^* \left( {\mb C}^N \right) \to QH^* \left( {\mb P}^{N-1} \right)$. We also identified the stable map compactification of the moduli space of nontrivial affine vortices of the lowest degree.

\subsection*{Organization}

The first half of this note is about general theory: Section \ref{section2} is on preliminaries of symplectic vortex equation, which includes the example of stable $N$-pairs. Section \ref{section3} is a review of previous work of Ziltener on affine vortices and its relation with vortex equation over a compact Riemann surface via the adiabatic limit. In Section \ref{section4} we consider the compactification of the moduli space of affine vortices, where we give some refinement of definitions of Ziltener. In Section \ref{section5} we review the (formal) definition of the quantum Kirwan map.

In the second half we restrict to the special case for $U(1)$-action on ${\mb C}^N$. In Section \ref{section6} we give a detailed description of the vortex bubbling phenomenon in the adiabatic limit. In Section \ref{section7} we use the adiabatic limit trick to give a classification of finite energy affine vortices which generalizes Taubes' classification for $N=1$; we also identify its moduli space and compute the associated quantum Kirwan map.

\subsection*{Acknowledgements}

The author would like to thank his advisor Professor Gang Tian for help and encouragement. He also would like to thank Chris Woodward, Sushmita Venugopalan and Fabian Ziltener for inspiring discussions.

\section{Symplectic vortex equation}\label{section2}

\subsection{Vortex equation}

Let $\left( M, \omega\right)$ be a symplectic manifold. Suppose $G$ is a compact Lie group acting on $M$ smoothly. Then for any $\xi \in {\mf g}$, the infinitesimal action of $\xi$ is the vector field ${\mc X}_\xi$ whose value at $p\in M$ is
\begin{align}
{\mc X}_\xi (p) = \left. {d\over dt}\right|_{t=0} \exp(t\xi) p. 
\end{align}
The action is Hamiltonian, if there exists a smooth map $\mu: M \to {\mf g}^*$ (the moment map) such that
\begin{align}
\iota_{{\mc X}_\xi} \omega = d \langle \mu, \xi \rangle_{\mf g}.
\end{align}
Also, we take a $G$-invariant, $\omega$-compatible almost complex structure on $J$. 

Let $\Sigma$ be a Riemann surface and we usually omit to mention its complex structure $j: T\Sigma \to T\Sigma$. Let $\Omega_\Sigma\in \Omega^2 (\Sigma)$ be a smooth area form. A {\it twisted holomorphic map} from $\Sigma$ to $M$ is a triple $\left( P, A, u \right)$, where $P \to \Sigma$ is a smooth principal $G$-bundle, $A$ is a smooth $G$-connection on $P$ and $u$ is a smooth section of the associated bundle $Y:= P\times_G M \to \Sigma$, satisfing the following {\bf symplectic vortex equation}:
\begin{align}\label{vortexeqn}
\left\{ \begin{array}{ccc}  \ov\partial_A u & = & 0;\\
                             \Lambda F_A + \mu(u) & = & 0.          \end{array}\right.
\end{align}
Here $\ov\partial_A u \in \Gamma\left( \Sigma, \Omega^{0, 1}\otimes u^*T^V Y \right)$ where $T^VY \to Y$ is the vertical tangent bundle; $\Lambda: \Omega^2(\Sigma) \to \Omega^0(\Sigma)$ is the contraction with respect to the area form $\Omega_\Sigma$; and for the second equation to make sense, we identify ${\mf g}$ with ${\mf g}^*$ via an ${\rm Ad}$-invariant inner product on ${\mf g}$. With respect to a local trivialization $P|_U = U \times G$ and a local holomorphic coordinate $z= s+ i t$ on $U$, $u$ corresponds to a map $\phi: U \to M$ and $A = d + \Phi ds + \Psi dt$, the equation (\ref{vortexeqn}) reads
\begin{align}\label{24}
\left\{\begin{array}{ccc}  \partial_s u - {\mc X}_{\Phi}(u) + J(u) \left( \partial_t u - {\mc X}_{\Psi}(u) \right) & = & 0;\\
                           \left( \partial_s \Psi - \partial_t \Phi + [ \Phi, \Psi]\right) ds dt  + \mu(u) \Omega_\Sigma & = & 0.                  \end{array} \right.
\end{align}

A solution $(P, A, u)$ is sometimes called a {\bf twisted holomorphic map} from $\Sigma$ to $M$. We say that two solutions $(P, A, u)$ and $(P', A', u')$ are equivalent, if there is a bundle isomorphism $\rho: P' \to P$ which lifts the identity map on $\Sigma$, such that $\rho^*(A, u) = (A', u')$. Denote by ${\mc G}(P)$ the group of smooth gauge transformations, which consists of smooth maps $g: P \to G$ with $g(p h) = h^{-1} g(p) h$. It acts on the space of solutions on the right, by 
\begin{align}
g^* \left( A, u \right) = \left( g^*A , g^{-1} u \right).
\end{align}
If written in a coordinate form as in (\ref{24}), $g$ corresponds to a smooth map $g: U \to G$ and the action is given by 
\begin{align}
g^* \left( d + \Phi ds + \Psi dt, \phi \right) = \left( d+ \left( {\rm Ad}_g^{-1} \Phi - g^{-1} \partial_s g\right) ds + \left( {\rm Ad}_g^{-1} \Psi - g^{-1} \partial_t g \right) dt , g^{-1} \phi \right). 
\end{align}
The Lie algebra of ${\mc G}(P)$ is the space of smooth sections of the vector bundle $P\times_{{\rm ad}} {\mf g}$ and for any section $s$, the infinitesimal action of $s$ is 
\begin{align}
{\mc X}_s \left( A, u \right) = \left( - d_A s,  - {\mc X}_s(u)  \right).
\end{align}

The energy of a twisted holomorphic map $(P, A, u)$ is given by the Yang-Mills-Higgs functional
\begin{align}
\mc{YMH}(A, u)= {1\over 2} \left( \left\| F_A\right\|_{L^2}^2 + \left\| \mu(u) \right\|_{L^2}^2 + \left\| d_A u \right\|_{L^2}^2 \right).
\end{align}
Here the $L^2$-norms are defined with respect to the Riemannian metric on $M$ determined by $\omega$ and $J$, and the Riemannian metric on $\Sigma$ determined by $\Omega_\Sigma$ and $j$.

\subsection{Example: holomorphic $N$-pairs}\label{subsection22}

Let $M= {\mb C}^k$ and let $G= U(k)$ which acts on ${\mb C}^k$ via the standard linear action. For the symplectic form
\begin{align}
\omega = \sum_{i=1}^k dx_i \wedge d y_i = {\sqrt{-1} \over 2} \sum_{i=1}^k dz_i \wedge d\ov{z}_i
\end{align}
a moment map is
\begin{align}
\mu(z_1, \ldots, z_k ) = -{\sqrt{-1} \over 2} \left( \sum_{i=1}^k z_i \otimes \ov{z}_i^T - \tau I_k \right) \in {\mf u}(k) \simeq {\mf u}(k)^*.
\end{align}
If $(P, A, u)$ is a twisted holomorphic map from $\Sigma$ to ${\mb C}^k$, then the associated bundle $E:= P\times_{U(k)} {\mb C}^k$ is a complex vector bundle with a Hermitian metric such that $P$ is the unitary frame bundle of $E$. The $(0, 1)$-component of $A$ defines a holomorphic structure on $E$ and $A$ is then the Chern connection determined by $\ov\partial_A$ and the Hermitian metric. The section $u: P \to {\mb C}^k$ corresponds to a holomorphic section of the bundle $\left( E, \ov\partial_A \right)$. A pair $(E, \ov\partial_A, u )$ with $\ov\partial_A u $ is called a rank $k$ holomorphic pair.

More generally, we make take $N$ copies of ${\mb C}^k$ and $U(k)$ acts on the $N$ copies in a diagonal way, and the moment map is the sum of the $N$ moment maps. In this case, a twisted holomorphic map corresponds to a rank $k$ holomorphic vector bundle $E$ with $N$ holomorphic sections, which is called a rank $k$ holomorphic $N$-pair. 

In the following we will only care about the abelian case, i.e., $k=1$. We can take $\tau = 1$ without loss of generality. A holomorphic $N$-pair $\left( {\mc L}; \varphi_1, \ldots, \varphi_N \right)$ is called {\bf stable}, if at least one of $\varphi_j$ is nonzero. 

We have the following important theorem, which is a special case of the celebrated {\it Hitchin-Kobayashi correspondence} (cf. \cite{Mundet_Hitchin_Kobayashi}).
\begin{thm}\cite{Bradlow_stable_pairs}\label{thm21} For any compact Riemann surface $\Sigma$ with any smooth area form $\Omega_\Sigma$ with ${\rm Area}\Omega_\Sigma > 4\pi d$, for any stable rank 1 holomorphic $N$-pair $\left( {\mc L}; \varphi_1, \ldots, \varphi_N \right)$ over $\Sigma$ with ${\rm deg} {\mc L}=d$, there exists a unique smooth Hermitian metric $H$ which solves the vortex equation, i.e., the following equation is satisfied:
\begin{align}
 F_H  - {\sqrt{-1} \over 2} \left( \sum_{j=1}^N \left| \varphi_j \right|_H^2 - 1 \right) \Omega_\Sigma = 0.
\end{align}
Here $F_H$ is the Chern connection of $\left( {\mc L}, H \right)$.
\end{thm}

\section{Affine vortices}\label{section3}

We now restrict to the case $\Sigma = {\mb C}$ and $\Omega_{\mb C} = ds \wedge dt$ the standard area form. We call a solution to the vortex equation (\ref{vortexeqn}) in this case an {\it affine vortex}. All $G$-bundles over ${\mb C}$ are trivial and isomorphisms between them are all isotopic, so we will work solely with ``the'' trivial bundle $P = {\mb C}\times G$. So a connection $A$ will be written canonically as $d + \alpha$ with $\alpha \in \Omega^1( {\mb C}, {\mf g})$ and the section $u$ corresponds canonically to a map $u: {\mb C} \to M$. A gauge transformation is then a map $g: {\mb C}\to G$.

The general theory for vortices over ${\mb C}$ initiated from the paper of Gaio-Salamon \cite{Gaio_Salamon_2005} and a lot of analytic framework has been settled down by Fabian Ziltener in \cite{quantumkirwan}. The algebraic theory of these objects are also studied in \cite{chris_quantum_kirwan}.

First of all, we make several assumptions on the manifold $(M,\omega)$ and the action. They are satisfied, for example, in the case of Subsection \ref{subsection22}.

\begin{hyp} We assume

\begin{enumerate}
\item $(M, \omega)$ is aspherical, i.e., for any embedded sphere $S^2\subset M$,
\begin{align}
\int_{S^2} \omega = 0.
\end{align}

\item
The moment map $\mu$ is proper, and $0$ is a regular value, such that the restriction of the $G$-action on $\mu^{-1}(0)$ is free. 

\item There exists a $G$-invariant, $\omega$-compatible almost complex structure $J$ such that the tuple $\left( M, \omega, J, \mu\right)$ is convex at infinity. This means there exists a proper $G$-invariant function $f: M \to [0, \infty)$ and a constant $C>0$ such that for any $(x, v) \in TM$ with $f(x) \geq C$, 
\begin{align}
\omega( \nabla_v \nabla f (x) , J(x) v) - \omega ( \nabla_{Jv} \nabla f (x), v) \geq 0,\ \omega( {\mc X}_{\mu(x)}, \nabla f ) \geq 0.
\end{align}
\end{enumerate}
\end{hyp}

We give two examples of affine vortices.
\begin{example} An affine vortex is called trivial if it is equivalent to $(A, u)$ where $A = d $ is the trivial connection on the trivial bundle, and $u: {\mb C} \to M$ is a constant map with value in $\mu^{-1}(0)$. In particular, an affine vortex is trivial if and only if it has zero energy.
\end{example}

\begin{example}\label{taubesexample} In the case $M= {\mb C}$, $U(1)$ acts on ${\mb C}$ by complex multiplication with moment map $\mu(z) = -{\sqrt{-1} \over 2} (|z|^2-1)$, Taubes (see \cite{Taubes_vortex}, \cite{Jaffe_Taubes}) classified all planary vortices with finite energy. More precisely, for any ``vortex number'' $d>0$ and for any $d$-tuple of unordered points $z_1, \ldots, z_d \in {\mb C}$, there is a unique solution (up to gauge) which is of the form $$(A, u) = \left( d - \partial h + \ov\partial h, e^{-h} (z-z_1)\cdots (z-z_d)\right).$$ Here $h$ is the unique solution to the {\it Kazdan-Warner equation} over ${\mb C}$:
\begin{align}
\Delta h + {1\over 2} \left( e^{- 2h} \prod_{j=1}^d |z- z_j|^2 - 1 \right) = 0.
\end{align}
\end{example}

There are two natural classes of symmetry of ${\mb C}$, the translations and rotations, with respect to which the equation is invariant. But there exsits solutions which have infinitely many rotational symmetry, which will result in non-smooth moduli space. Hence we won't identify two solutions if they differ by a rotation.
\begin{defn}\label{defn34}
An isomorphism from $(A_1, u_1)$ to $(A_2, u_2)$ is a pair $(t, g)$, where $t: {\mb C}\to {\mb C}$ is a translation, which lifts naturally to a bundle map between the trivial $G$-bundles, and $g$ is a gauge transformation $g: {\mb C}\to G$, such that $g^* t^* \left( A_1, u_1 \right) = \left( A_2, u_2 \right)$.
\end{defn}

We have the regularity modulo gauge tranformation. 
\begin{prop}\cite[Proposition D.2]{Ziltener_thesis} For any solution $(A, u)$ of class $W^{1, p}_{loc}$, there exists a gauge tranformation $g$ of class $W^{2, p}_{loc}$ such that $g^*(A, u)$ is smooth.
\end{prop}
From now on any solution will be assumed to be smooth unless otherwise mentioned.

The next important property for affine vortices is its behavior near infinity. First we look at the decay of energy density.
\begin{prop}\cite[Corollary 4]{Ziltener_Deday} Suppose $(A, u)$ is an affine vortex with finite energy. Then for any $\epsilon>0$ there exists $C_\epsilon>0$ such that
\begin{align}
e_{A, u}(z) \leq C_\epsilon |z|^{-4 + \epsilon},\ \forall |z|\geq 1.
\end{align}
\end{prop}

Then we have the assymptotic behavior of vortices in suitable gauge. A solution $(A, u)$ is {\bf in radial gauge} if for $|z|$ large, $A =  d + \eta(z) d\theta$.
\begin{prop}\cite[Proposition 11.1]{Gaio_Salamon_2005} If $(A, u)$ is an affine vortex with finite energy in radial gauge with $A= d + \xi(z) d\theta$ for $|z|$ large. Then there exists a $W^{1, 2}$-map $x: S^1 \to \mu^{-1}(0)$ and an $L^2$-map $\eta: S^1 \to {\mf g}$ such that $x'(\theta) + {\mc X}_{\eta(\theta)}(x(\theta)) = 0$ and 
\begin{align}
\lim_{ r \to \infty} \sup_{\theta} d \left( u( re^{i\theta}), x(\theta) \right) = 0,\ \lim_{r\to \infty} \int_0^{2\pi} \left| \xi  ( r e^{i \theta}) - \eta(\theta) \right|^2 d \theta = 0.
\end{align}
\end{prop}

To state the next proposition we introduce some notations. Consider the standard embedding ${\mb C}\to S^2$. An {\bf extension} of $(A, u)$ is a triple $(\wt{P}, \iota, \wt{u})$, where $\wt{P}\to S^2$ is a topological $G$-bundle, $\iota: {\mb C}\times G \to \wt{P}$ is a bundle map which descends to the inclusion ${\mb C}\to S^2$, and $\wt{u}: \wt{P}\to M$ is an equivariant continuous map such that $\wt{u}\circ \iota = u$.

\begin{prop}\cite{Gaio_Salamon_2005}\cite{quantumkirwan}\label{prop39}
If $\mc{YMH}(A, u) < \infty$, then there exists a unique extension $(\wt{P}, \iota, \wt{u})$ of $(A, u)$ onto the sphere, such that $\wt{u}(\wt{P}_\infty) \subset \mu^{-1}(0)$, where $\wt{P}_\infty$ is the fibre of $\wt{P}$ at $\infty \in S^2$.
\end{prop}

The extension $(\wt{P}, \iota, \wt{u})$ defines an equivariant homology class $\wt{u}_*[S^2] \in H_2^G(M; {\mb Z})$. Indeed, the Yang-Mills-Higgs functional of $(A, u)$ is equal to the paring $\langle [\omega- \mu], \wt{u}_*[S^2]\rangle$, where $[\omega-\mu]$ is the equivariant cohomology class defined by the equivariantly closed two form $\omega -\mu$.

Identify $u$ topologically with $v: {\mb D} \to M$ which extends continuously to the boundary. Then the extension of $(A, u)$ to $S^2$ is given by $g_\infty: S^1 \to G$ and $x_\infty \in \mu^{-1}(0)$ such that $v|_{\partial {\mb D}}( e^{i \theta})  = g_\infty(e^{i \theta}) x_\infty$. Trivialize $v^* TM$ on ${\mb D}$, then the loop $g_\infty$ is identifed with a loop of invertible matrices. The degree of $\det g_\infty$ is defined to be the {\bf Maslov index} of the solution $(A, u)$, denoted by ${\rm ind}_\mu(A, u)$ (see \cite[Definition 2.6]{quantumkirwan} for details).

\subsection{Fredholm theory of affine vortices}\label{subsection31}

Let $z$ be the standard coordinate on ${\mb C}$. Let $\rho(z) = ( 1+ |z|^2)^{1\over 2}$ for $z\in {\mb C}$. For $\delta \in {\mb R}$, $p>1$, consider the Banach spaces (over complex numbers)
\begin{align}
W^{k, p}_{\delta, \mf{euc}}:= \left\{ u \in W^{k, p}_{loc}({\mb C})\ |\  \rho^\delta u \in W^{k, p}({\mb C})\right\}.
\end{align}
And we denote $L^p_\delta= W^{0, p}_{\delta, \mf{euc}}$. 

For a smooth solution $(A, u)$, the space of infinitesimal deformations is described as follows. Regard $u$ as a smooth map from ${\mb C}$ to $M$. The connection $A$ induces a connection on the bundle $u^* TM $ and the bundle $T^*{\mb C}\otimes {\mf g}$, both denoted by $\nabla^A$. For $(V, \alpha)\in W^{1, p}_{loc} \left( {\mb C}, u^*TM \oplus T^*{\mb C}\otimes {\mf g} \right)$, define
\begin{align}
\left| (V, \alpha) \right|_{W_{p, \delta}} := \left| V \right|_{L^\infty} + \left| \nabla^A V \right|_{L^p_\delta} + \left| \nabla^A \alpha \right|_{L^p_\delta} + \left| d\mu(V) \right|_{L^p_\delta} + \left| d\mu( J V) \right|_{L^p_\delta} + \left| \alpha \right|_{L^p_\delta}.
\end{align}
And define $W_{p, \delta}\subset W^{1, p}_{loc}$ be the subspace of vectors with finite $W_{p, \delta}$-norm. If $A= d + \alpha_0$, then the linearization of the gauge tranformation is $h \mapsto  \left( - dh + [\alpha_0, h], -{\mc X}_h \right)$, whose adjoint is the map
\begin{align}
\begin{array}{ccc}
W_{p, \delta} & \to & L_\delta^p \left( {\mb C}, {\mf g} \right)\\
 \left( V, \alpha \right) & \mapsto & - d^* \alpha - \Lambda [ * \alpha_0, \alpha] - d\mu ( J V)
\end{array}
\end{align}
Then it is easy to see, the linearization at such $(A, u)$ is a bounded linear operator
\begin{align}\label{eqn39}
\begin{array}{cccc}
D_{A, u} : & W_{p, \delta} & \to & L^p_\delta\\
           & \left( V, \alpha \right) & \mapsto & \left( \begin{array}{c} \left( \nabla^A V \right)^{0, 1} + {1\over 2} (\nabla_V J ) \circ d_A u \circ j + {\mc X}_{\alpha}^{0, 1} \\
           d \alpha + [ \alpha_0 , \alpha] + d\mu(V) \cdot ds dt\\
           - d^* \alpha - \Lambda [* \alpha_0, \alpha] - d\mu( JV)
           \end{array} \right)
           \end{array}
\end{align}

\begin{prop}\cite{quantumkirwan}
There exists $p_0>2$ such that for all $p\in (2, p_0)$ and $\delta \in \left( 1-{2\over p}, 2- {2\over p} \right)$, $D_{A, u}$ is Fredholm and ${\rm index} D_{A, u}= 2 {\rm ind}_\mu(A, u) + {\rm dim} M - 2 {\rm dim}G$.
\end{prop}

\subsection{The adiabatic limit}

For a compact Riemann surface $\Sigma$, fix a smooth area form $\Omega_\Sigma \in \Omega^2(\Sigma)$. Let $\lambda >0$ be a real number. A {\bf $\lambda$-twisted holomorphic map} from $\Sigma$ to $M$ (or a $\lambda$-vortex) is a solution to the vortex equation (\ref{vortexeqn}) with the area form $\Omega_\Sigma$ replaced by $\lambda^2 \Omega_\Sigma$, i.e.,
\begin{align}\label{lambdavortex}
\left\{ \begin{array}{ccc} \ov\partial_A u & = & 0;\\
                           F_A + \lambda^2 \mu(u) \Omega_\Sigma & = & 0. 
\end{array} \right.
\end{align}
Its energy is defined in the same way using the area form $\lambda^2 \Omega_\Sigma$. If we agree that all Sobolev norms appearing in the following are taken with respect to the fixed area form $\Omega_\Sigma$, then the energy of a $\lambda$-twisted holomorphic map is given by 
\begin{align}
\mc{YMH}^\lambda(P, A, u) = {1\over 2} \left(  \left\| d_A u \right\|_{L^2}^2 + \lambda^{-2} \left\| F_A \right\|_{L^2}^2 + \lambda^2 \left\| \mu( u ) \right\|_{L^2}^2 \right).
\end{align}

What is of interest is the limit process $\lambda \to \infty$, which is called the {\bf adiabatic limit process}. If we fix the topological type of the vortex (which implies the uniform bound on the energy), then we see that $\left\| \mu( u ) \right\|_{L^2} \to 0$ as $\lambda \to \infty$. Indeed, $\mu( u )$ will converge to zero except for finitely many points in $\Sigma$, and those points are where nontrivial affine vortices bubble off.

More precisely, suppose $\lambda_k$ is a sequence of real numbers diverging to infinity, and $(A_k ,u_k)$ is a sequence $\lambda_k$-vortices (i.e., solutions to (\ref{lambdavortex})). The energy density function for $(A_k, u_k)$ will be
\begin{align}
e_k(z) = {1\over 2} \left(  \left| d_{A_k} u_k(z) \right|^2 + \lambda_k^2 \left| \mu(u_k (z) ) \right|^2 \right) . 
\end{align}
If the sequence of functions $e_k$ is not uniformly bounded, and suppose there exists a sequence of point $p_k \in \Sigma$ such that
\begin{align}
\limsup_{k \to \infty} c_k := \limsup_{k \to \infty} e_k(p_k) =  \limsup_{k \to \infty} \sup_\Sigma e_k  = + \infty.
\end{align}
Then in \cite{Gaio_Salamon_2005} it was shown that in the following three possibilities (for suitable subsequences) corresponding bubbles will appear:
\begin{enumerate}
\item $\lim_{k \to \infty} \lambda_k^{-2} c_k = +\infty$. In this case, a nontrivial holomorphic sphere in $M$ will bubble off.

\item $0< \lim_{k\to \infty} \lambda_k^{-2} c_k < +\infty$. In this case, a nontrivial planary vortex will bubble off.

\item $\lim_{k \to \infty} \lambda_k^{-2} c_k = 0$. In this case, a nontrivial holomorphic sphere in the symplectic quotient will bubble off.
\end{enumerate}
This is called the {\bf bubbling zoology} of the adiabatic limit. Since we have assumed that $(M, \omega)$ is aspherical, the first bubble type won't appear. In particular we have
\begin{align}\label{eqn314}
\limsup_{k \to \infty} \lambda_k^{-2} \left\| e_k \right\|_{L^\infty} < \infty.
\end{align}

\begin{rem}
In the general situation, given an arbitrary sequence $(A_i, u_i)$, we don't know {\it a priori} where the energy density will blow up and what type of bubbles may appear. But in the case of K\"ahler targets, using the algebraic description of vortices (e.g., stable $N$-pairs) provided by the Hitchin-Kobayashi correspondence, we observe that the cause of the energy concentration is governed in the algebraic side. Using this property we can manipulate the energy concentration for the adiabatic limit process and construct affine vortex bubbles. This is what we do in the last two sections of this paper for the case $X= {\mb C}^N$.
\end{rem}

\section{Degenerations of affine vortices}\label{section4}

In the case studied by Taubes, we have already seen one type of degeneration of affine vortices. Namely, vortices may ``split'': the relative distance between points in $\ora{z}\in {\rm Sym}^d{\mb C}$ may diverge to infinity, which means a sequence of vortices with vortex number $d$ can split into up to $d$ nontrivial vortices. There are two other types of degenerations. If the symplectic quotient $\ov{M}$ allows nontrivial holomorphic spheres, then the energy of a sequence of affine vortices can concentrate at infinity and a holomorphic sphere in $\ov{M}$ can bubble off. Another possibility is sphere bubbles in the interior. But since we have assumed that $M$ is aspherical, this type of bubbles can only appear when marked points coming together, i.e., ghost bubbles.

The bubbling phenomenon has been studied in \cite{quantumkirwan}. And more generally, if we allow arbitrarily many marked points, the we can use the language of {\it complexified multiplihedron} to describe the moduli of stable vortices and its topology, as did in \cite{chris_quantum_kirwan}. Here we only consider the case of at most one interior marked point, hence no need to introduce the formal language.

\subsection{Stable maps modelled on a rooted tree}

A rooted tree $T= (V, E, R)$ is a tree $T = (V, E)$ with a distinguished vertex $R\in V$(the {\it root}). For a rooted tree $T$ the edges are automatically oriented towards the root. We define the {\it depth} $d_T: V \to {\mb Z}_{\geq 0} \cup \{\infty\}$ such that
\begin{enumerate}
\item $d_T(R) = \infty$;

\item $v_1 E v_2 \in E \Longrightarrow d_T(v_1) + 1 = d_T (v_2)$ or $v_2 = R$;

\item If there is no $v_1$ such that $v_1 E v_2 \in E$, then $d_T(v_2) = 0$.
\end{enumerate}

An $n$-labelling of a rooted tree $T$ is a map $\rho: \{\alpha_1, \ldots, \alpha_n\} \to V$.

\begin{defn}
An $(n,1)$-marked genus zero stable map to $\ov{M}$ modelled on an $n$-labelled rooted tree $(T, \rho)= (V, E, R; \rho)$ is a tuple
$$ \left( {\bf u}, {\bf w}\right):= \left( \{ u_{v_i} \}_{v_i \in V}, \{w_{i_1 i_2}\}_{v_{i_1} E v_{i_2} \in E}, \{w_{\alpha_j}\}_{j=1, \ldots, n}         \right)$$
where 
\begin{enumerate}
\item $u_{v_i}: {\mb C} \to \ov{M}$ is a holomorphic map with finite energy (hence extends to a holomorphic sphere);

\item $w_{i_1 i_2} \in {\mb C}$ and $w_{\alpha_j} \in {\mb C}$.
\end{enumerate}
They are subject to the following conditions:
\begin{enumerate}
\item For each $v_{i_1} E v_{i_2} \in E$, $u_{v_{i_1}}(\infty) = u_{v_{i_2}}( w_{i_1 i_2}) \in \ov{M}$;

\item For each $v_{i_2} \in V$, the points $w_{i_1 i_2}$ for all $v_{i_1} E v_{i_2} \in E$ and $z_{\alpha_j}$ for all $\rho(\alpha_j) = v_{i_2}$ are all distinct;

\item If $u_{v_{i_2}}$ is a constant map, then $\#\{ v_{i_1} \in V\ |\ v_{i_1} E v_{i_2} \in E\} + \#\{ \alpha_j \ |\ \rho(\alpha_j) = v_{i_2} \}\geq 2$.
\end{enumerate}
For each $v_i \in V$, we define $Z_{v_i}:= \left\{ w_{v_i' v_i} \in {\mb C}\ |\  v_i' E v_i \in E \right\} \cup \left\{ w_{\alpha_j}\ |\ \rho( \alpha_j) = v_i \right\}$.

\end{defn}
This refined notion of stable maps is used because the domain of an affine vortex has a canonical marked point $\infty$. Then each stable map $\left( {\bf u}, {\bf w} \right)$ defined as above, there is a dinstinguished marked point $\infty$ on the component corresponding to $R$. For each edge $v_1 E v_2$, the coordinate of the node on the component corresponding to $v_1$ is automatically $\infty$. We can define various notions (such as isomorphisms and convergence) of stable maps as did in \cite[Chapter 5]{McDuff_Salamon_2004}, with the restriction that all tree maps should be maps between rooted trees, and we should only use affine linear transformations instead of arbitrary M\"obius transformations.

\subsection{$(0, 1)$-marked and $(1, 1)$-marked stable affine vortices}\label{subsection42}

We now describe the objects we will use to compactify the moduli space of affine vortices discussed in Section \ref{section3}.

\begin{defn}
An {\bf admissible pair of rooted trees} is a pair $\left( \wt{T}, T \right)$ where $\wt{T}= \left( \wt{V}, \wt{E}, R \right)$ is a rooted tree, $T\subset \wt{T}$ is a rooted subtree and we allow $T =\emptyset$ such that the following conditions hold
\begin{enumerate}
\item If $T= \emptyset$, then $\wt{T}$ has a single vertex $R$;

\item If $T= (V, E, R)\neq \emptyset$, then $\wt{V}\setminus V$ consists of vertices of depth zero;
\end{enumerate}
A {\bf labelled admissible pair of rooted trees} is an admissible pair of rooted trees $\left( \wt{T}, T \right)$ together with a vertex $V_0 \in \wt{V}\setminus V$, denoted by $\left( \wt{T}, T; V_0 \right)$.

\bigskip

If $T\neq \emptyset$, then the set of edges $\beta_k E v_i \in \wt{E}$ for $\beta_k \in \wt{V}\setminus V$ and $v_i\in V$ induces a labelling $\rho_{\wt{T}, T}: \wt{V} \setminus V \to V$, making $\left( T, \rho_{\wt{T}, T} \right)$ a rooted tree with a $\# (\wt{V}\setminus V)$-labelling.

\bigskip 

More generally, if $T\neq \emptyset$, and $T' \subset T$ is a rooted subtree (which is allowed to be empty), then denote by ${\mc B}_{\wt{T}, T'}:= \{B_1, \ldots, B_t\}$ the set of connected components of $\wt{T}\setminus T'$. Each element $B_l \in {\mc B}_{\wt{T}, T'}$ represent a rooted tree $\wt{T}_{B_l}$ so that $\left( B_l, B_l \cap T \right)$ is an admissible pair of rooted trees. If $V_0 \in B_l$, then $ \left( B_l, B_l \cap T; V_0 \right)$ is a labelled admissible pair of rooted trees.
\end{defn}

\begin{defn} A $(0, 1)$-marked stable affine vortex modelled on an admissible pair of rooted trees $\left( \wt{T}, T \right)$ is a tuple
\begin{align}
\wt{\bf W}:= \left( \left( {\bf u}, {\bf w} \right); \left\{ W_{\beta_k}\right\}_{\beta_k \in \wt{V}\setminus V} \right)	
\end{align}
where
\begin{enumerate}
\item If $T \neq \emptyset$, then $ \left( {\bf u}, {\bf w}\right)$ is an $(s, 1)$-marked genus zero stable map to $\ov{M}$ modelled on the labelled rooted tree $\left( T, \rho_{\wt{T}, T} \right)$; if $T= \emptyset$ then $\left({\bf u}, {\bf w} \right)$ is empty.

\item {\bf Stability.} For each $\beta_k \in \{ \beta_1, \ldots, \beta_s\}= \wt{V}\setminus V$, $W_{\beta_k}$ is a {\it nontrivial} affine vortex.
\end{enumerate}
They are subject to the following constrains:
\begin{itemize}
\item If $T \neq \emptyset$, then for each $\beta_k\in \wt{V}\setminus V$, $ev_\infty\left( W_{\beta_k} \right) = ev_k(	{\bf u}) \in \ov{M}$.
\end{itemize}
\end{defn}

\begin{defn} A $(1, 1)$-marked stable affine vortex modelled on a labelled admissible pair of rooted trees $\left( \wt{T}, T; V_0 \right)$ is a tuple
\begin{align}
\wt{\bf W}:= \left( \left( {\bf u}, {\bf w} \right); \left\{ W_{\beta_k} \right\}_{\beta_k \in\wt{V}\setminus V } ; \{z_0\} \right)
\end{align}
where
\begin{enumerate}
\item If $T \neq \emptyset$, then $ \left( {\bf u}, {\bf w}\right)$ is an $(s, 1)$-marked genus zero stable map to $\ov{M}$ modelled on the labelled tree $\left( T, \rho_{\wt{T}, T} \right)$; if $T= \emptyset$ then $\left({\bf u}, {\bf w} \right)$ is empty.

\item For each $\beta_k \in \{ \beta_1, \ldots, \beta_s\}= \wt{V}\setminus V$, $W_{\beta_k}$ is an affine vortex;

\item $z_0 \in {\mb C}$.
\end{enumerate}
They are subject to the following constrains:
\begin{enumerate}
\item {\bf Stability.} If $W_{\beta_k}$ is trivial then $V_0 = \beta_k$.

\item If $T \neq \emptyset$, then for each $\beta_k\in \wt{V}\setminus V$, $ev_\infty\left( W_{\beta_k} \right) = ev_k(	{\bf u}) \in \ov{M}$.
\end{enumerate}
\end{defn}

We will call the marked point $z_0$ an {\bf interior marked point} of $\wt{\bf W}$.

The following picture illustrates a typical marked stable affine vortex. Here the ``tear drops'' represent the vortices $W_k$ and the spheres represent the components mapped into $\mu^{-1}(0)$.
\begin{figure}[htbp]
\centering
\includegraphics[scale=0.70]{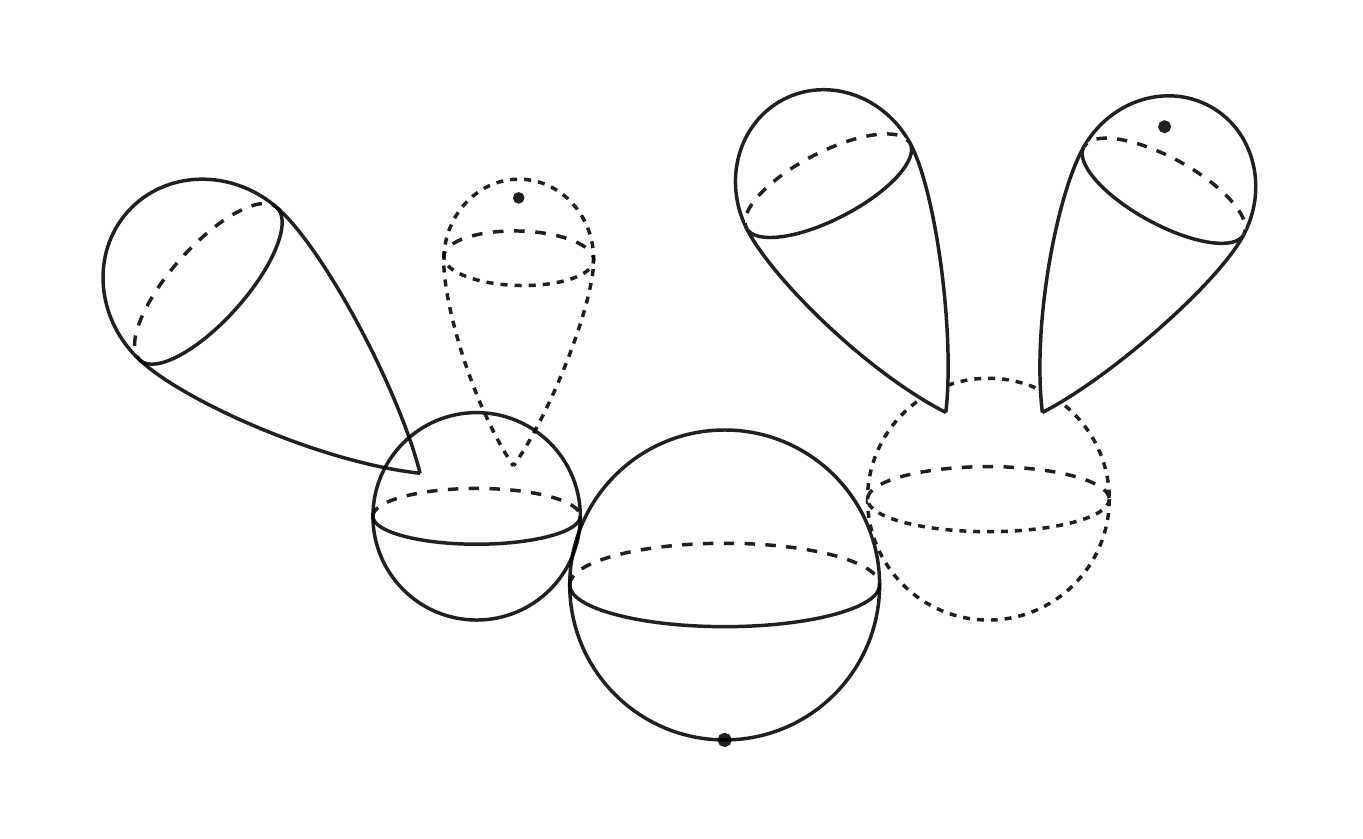}
\caption{A typical marked stable affine vortex.}
\label{figure1}
\end{figure}

\begin{defn}\label{defn45}
An isomorphism between two $(1, 1)$-marked planary stable vortices 
$$\wt{\bf W}:= \left( \left( {\bf u}, {\bf w} \right); \left\{ W_{\beta_k} \right\}_{\beta_k \in\wt{V}\setminus V}; \{z_0\}\right)
$$
and
$$\wt{\bf W}':= \left( \left( {\bf u}', {\bf w}' \right); \left\{ W_{\beta_k}'\right\}_{\beta_k \in \wt{V} \setminus V}; \{z'_0\} \right)$$ 
modelled on the same labelled admissible pair of rooted trees $\left( \wt{T}, T; V_0 \right)$ is a tuple
\begin{align}
\left( f, \left\{ t_{\beta_k}, g_{\beta_k} \right\}_{\beta_k \in \wt{V}\setminus V}, \{ \phi_{v_i} \}_{v_i\in V} \right) 
\end{align}
where
\begin{enumerate}
\item $f: (\wt{T}, T; V_0) \to \left( \wt{T}, T;V_0\right)$ is an automorphism;

\item For each $\beta_k\in \wt{V}\setminus V$, $t_{\beta_k}$ is a translation on ${\mb C}$ and $g_{\beta_k}: {\mb C} \to G$ is a gauge transformation;

\item For each $v_i \in V$, $\phi_{v_i}$ is an affine linear transformation and if $T \neq \emptyset$, then $\left( f|_T, \{ \phi_{v_i}\}_{v_i\in V} \right)$ is an isomorphism between $({\bf u}, {\bf w})$ and $( {\bf u}', {\bf w}' )$ as $(s, 1)$-marked genus zero stable maps modelled on the $\# \left( \wt{V}\setminus V \right)$-labelled rooted tree $\left( T, \rho_{\wt{T}, T} \right)$.
\end{enumerate} 
They must satisfy the following conditions:
\begin{enumerate}
\item For each $\beta_k \in \wt{V} \setminus V$, $(t_{\beta_k}, g_{\beta_k})$ is an isomorphism from $W_{\beta_k}$ to $W_{f(\beta_k)}'$ (see Definition \ref{defn34});

\item $t_{V_0}( z_0') = z_0$.
\end{enumerate}
\end{defn}

\bigskip

For $k = 0, 1$, we define the homology class of a $(k, 1)$-marked stable affine vortex $\wt{\bf W}$ to be the sum of the homology classes of each of its components, which is an element in $H_2^G \left( M; {\mb Z} \right)$. This only depends on its isomorphism class. For $A\in H_2^G(M; {\mb Z})$ we denote by $\ov{\wt {\mc M}}_{1, 1}^{\mb A}(M, A)$ be the category of all $(1, 1)$-marked stable affine vortices of homology class $A$ and the morphisms are isomorphisms between the objects. Denote by $\ov{\mc M}^{\mb A}_{1, 1}(M, A)$ the space of isomorphism classes of $(1,1)$-marked stable affine vortices of homology class $A$. There is a well-defined evaluation map
\begin{align}
ev_\infty: \ov{\mc M}^{\mb A}_{1, 1} (M, A) \to \ov{M}.
\end{align}

\subsection{Degeneration of affine vortices}\label{subsection43}

Now we describe the topology on the moduli space of stable affine vortices. We first give the definition the convergence of a sequence of affine vortices to a stable affine vortex, which essentially coincide with the definition in \cite{quantumkirwan}.

By a theorem of Guillemin and Sternberg, there is a neighborhoof $U_\epsilon$ of $\mu^{-1}(0) \subset M$, which is (canonically) symplectomorphic to $\mu^{-1}(0) \times {\mf g}_\epsilon^*$, where ${\mf g}_\epsilon^*$ is an $\epsilon$-ball of ${\mf g}^*$ centered at the origin with respect to some biinvariant inner product, such that the moment map $\mu$ restricted to $U_\epsilon$ is equal to the projection onto ${\mf g}_\epsilon^*$. Hence there is a well-defined map
\begin{align}
\pi_\mu: U_\epsilon \to \ov{M}.
\end{align}

\begin{defn}\label{defn36} Let $W^\nu= \left( A^\nu, u^\nu\right)$ be a sequence of affine vortices and let $$\wt{\bf W}= \left( ({\bf u}, {\bf w}); \left\{ W_{\beta_k}\right\}_{\beta_k \in \wt{V}\setminus V}\right)$$ be a $(0, 1)$-marked stable affine vortex modelled on an admissible pair of rooted trees $\left( \wt{T}, T \right)$. We say that the sequence $\left\{ W^\nu \right\}$ converges to $\wt{\bf W}$, if $E:= \lim_{\nu \to \infty} E(W^\nu) < \infty$ exists, and 
\begin{align}
E= E(\wt{\bf W})
\end{align}
and for each $v_i\in V$ and $\beta_k \in \wt{V}\setminus V$, there exists affine linear transformations $\phi_{v_i}^\nu$, $\phi_{\beta_k}^\nu$ such that the following conditions are satisfied:
\begin{enumerate}
\item For each $\beta_k\in \wt{V}\setminus V$, $\phi_{\beta_k}^\nu$ is a translation, and there exists a sequence of gauge transformations $g^\nu_{\beta_k}: {\mb C} \to G$ such that $ \left( g_{\beta_k}^\nu \right)^*\left( \phi_{\beta_k}^\nu \right)^* (A^\nu, u^\nu)$ converges to $W_{\beta_k}$ uniformly in any compact subset of ${\mb C}$;

\item For each $v_i \in V$, the sequence affine linear transformation $\phi_{v_i}^\nu(y) = a_{v_i}^\nu y + b_{v_i}^\nu$ with $a_{v_i}^\nu$ converges to infinity; 

\item For each $v_i \in V$, $ \left( \phi_{v_i}^\nu\right)^* \left( \mu \circ u^\nu \right)$ converges to zero uniformly on any compact subset of ${\mb C}\setminus Z_{v_i}$, and the map $ \left( \phi_{v_i}^\nu \right)^* \left( \pi_\mu \circ u^\nu \right)$ converges to $u_{v_i}$ uniformly on any compact subset of ${\mb C}\setminus Z_{v_i}$.

\item If $T \neq \emptyset$, then for any $\beta_k \in \wt{V}\setminus V$, the sequence of affine linear transformations $ \left( \phi^\nu_{\rho_{\wt{T}, T}(\beta_k)} \right)^{-1} \circ \phi^\nu_{\beta_k}$ converges to the constant map $w_{\beta_k}$ uniformly on any compact subset of ${\mb C}$;

\item For any $v_i E v_j \in E$, the sequence of affine linear tranformations $\left( \phi^\nu_{v_j} \right)^{-1} \circ \phi^\nu_{v_i}$ converges uniformly on any compact subset of ${\mb C}$ to the constant map $w_{v_i v_j}$.
\end{enumerate}
\end{defn}

\begin{defn}\label{defn37} Let $W^\nu= \left( A^\nu, u^\nu; z_0^\nu \right)$ be a sequence of $(1, 1)$-marked affine vortices and let $$\wt{\bf W}= \left( ({\bf u}, {\bf w}); \left\{ W_{\beta_k} \right\}_{\beta_k \in \wt{V}\setminus V}; \{z_0\}\right)$$ be a $(1, 1)$-marked stable affine vortex modelled on the labelled admissible pair of rooted trees $(\wt{T}, T; V_0)$. We say that the sequence $W^\nu$ converges to $\wt{\bf W}$, if $E:= \lim_{\nu \to \infty} E(W^\nu) < \infty$ exists, and 
\begin{align}
E= E(\wt{\bf W})
\end{align}
and for each $v_i\in V$ and $\beta_k \in \wt{V}\setminus V$, there exists affine linear transformations $\phi_{v_i}^\nu$, $\phi_{\beta_k}^\nu$ such that the following conditions are satisfied:
\begin{enumerate}
\item For each $\beta_k \in \wt{V}\setminus V$, $\phi_{\beta_k}^\nu$ is a translation, and there exists a sequence of gauge transformations $g^\nu_{\beta_k}: {\mb C} \to G$ such that $ \left( g_{\beta_k}^\nu \right)^*\left( \phi_{\beta_k}^\nu \right)^* (A^\nu, u^\nu)$ converges to $W_{\beta_k}$ uniformly in any compact subset of ${\mb C}$;

\item $\lim_{\nu\to \infty} \left(\phi_{V_0}^\nu \right)^{-1} (z_0^\nu) = z_0$;

\item For each $v_i \in V$, the sequence affine linear transformation $\phi_{v_i}^\nu(y) = a_{v_i}^\nu y + b_{v_i}^\nu$ with $a_{v_i}^\nu$ converges to infinity; 

\item For each $v_i \in V$, $ \left( \phi_{v_i}^\nu\right)^* \left( \mu \circ u^\nu \right)$ converges to zero uniformly on any compact subset of ${\mb C}\setminus Z_{v_i}$, and the map $ \left( \phi_{v_i}^\nu \right)^* \left( \pi_\mu \circ u^\nu \right)$ converges to $u_{v_i}$ uniformly on any compact subset of ${\mb C}\setminus Z_{v_i}$.

\item If $T \neq \emptyset$, then for any $\beta_k \in \wt{V}\setminus V$, the sequence of affine linear transformations $ \left( \phi^\nu_{\rho_{\wt{T}, T}(\beta_k)} \right)^{-1} \circ \phi^\nu_{\beta_k}$ converges to the constant map $w_{\beta_k}$ uniformly on any compact subset of ${\mb C}$;

\item For any $v_i E v_j \in E$, the sequence of affine linear tranformations $\left( \phi^\nu_{v_j} \right)^{-1} \circ \phi^\nu_{v_i}$ converges uniformly on any compact subset of ${\mb C}$ to the constant map $w_{v_i v_j}$.
\end{enumerate}
\end{defn}

Now the notion of convergence in the space of planary stable vortices can be easily defined:
\begin{defn}\label{defn38} Let $\wt{\bf W}^\nu = \left( \left( {\bf u}^\nu, {\bf w}^\nu \right);  \left\{ W_{\beta_k}^\nu\right\}_{\beta_k \in \wt{V}^\nu \setminus V^\nu} ; \{z_0^\nu\} \right)$ be a sequence of $(1, 1)$-marked stable affine vortices modelled on a sequence of labelled admissible pair of rooted trees $\left( \wt{T}^\nu, T^\nu; V_0^\nu \right)$, and 
$$\wt{\bf W}^\infty= \left( \left( {\bf u}^\infty, {\bf w}^\infty \right); \{W_{\beta_k}^\infty \}_{\beta_k\in \wt{V}^\infty \setminus V^\infty };\{ z_0^\infty \}\right)$$ 
be a $(1, 1)$-marked stable affine vortex modelled on the labelled admissible pair of rooted trees $\left( \wt{T}^\infty, T^\infty; V_0^\infty \right)$. We say that the sequence $\left\{ \wt{\bf W}^\nu \right\}$ converges to $\wt{\bf W}^\infty$ if the following condition holds:
\begin{enumerate}
\item If $T^\infty = \emptyset$, then for large $\nu$, $T^\nu = \emptyset$ and $\wt{\bf W}^\nu$ converges to $\wt{\bf W}^\infty$ in the sense of Definition \ref{defn37};

\item If $T^\infty \neq \emptyset$, then there exists a rooted subtree $T' \subset T^\infty$ such that
\begin{enumerate}
\item If $T' = \emptyset$ then for large $\nu$, $T^\nu =\emptyset$;

\item If $T' \neq \emptyset$, then for large $\nu$, $T^\nu \neq \emptyset$, $\# \left( \wt{V}^\nu\setminus V^\nu \right) = \# {\mc B}_{\wt{T}^\infty, T'}$, and the sequence of genus zero stable maps $\left( {\bf u}^\nu, {\bf w}^\nu \right)$ converges to the stable map $\left. \wt{\bf W}^\infty \right|_{T'}$ which is modelled on the rooted tree $T'$ with labelling $\rho_{\wt{T}^\infty, T'}$;

\item For large $\nu$, there is a bijection ${\mf s}: {\mc B}_{\wt{T}^\infty, T'} \to \wt{V}^\nu \setminus V^\nu$ and a unique $B_0$ such that $V^\nu_0 = {\mf s}( B_0)$; and such that for each $B_l \in {\mc B}_{\wt{T}^\infty, T'}$, if $B_l= B_0$, then $\left( W^\nu_{V_0}, z_0^\nu \right)$ converges to $\left. \wt{\bf W}^\infty\right|_{B_0}$ in the sense of Definition \ref{defn37}; if $B_l \neq B_0$, then $W^\nu_{{\mf s}(B_l)}$ converges to $\left. \wt{\bf W}^\infty \right|_{B_l}$ in the sense of Definition \ref{defn36}.
\end{enumerate}
\end{enumerate}
\end{defn}	

Now we can state the compactness theorem of Ziltener. We specialize to the case where we have at most one interior marked points.
\begin{thm}\cite[Theorem 3]{quantumkirwan}\label{thm49}
If $\wt{\bf W}^\nu$ is a sequence of $(k, 1)$-marked stable affine vortices with $k = 0$ or $1$, and $\limsup_{\nu \to \infty} E \left( \wt{\bf W}^\nu\right) < \infty$, then there is a subsequence and a $(k, 1)$-marked stable affine vortex $\wt{\bf W}^\infty$ such that the subsequence converges to $\wt{\bf W}^\infty$ in the sense of Definition \ref{defn38}.
\end{thm}

\begin{example}
Suppose we are in the case of {\it Example} \ref{taubesexample}. A sequence of vortices is prescribed by a sequence of points $\ora{z}^\nu \in {\rm Sym}^d {\mb C}$. Modulo translation, the bubbling is caused by the phenomenon that at least two points in $\ora{z}^\nu$ are separated infinitely far away. Then in the limit, $\ora{z}^\nu$ are divided into groups, points in the same group will stay within finite distance and points belonging to different groups will be infinitely far away. Hence the limit depends on a partition $ d= d_1 + \ldots + d_r$, a curve $C \in \ov{\mc M}_{0, r+1}$, and for each $i \in \{1, \ldots, r\}$, an element $\ora{z}_i \in {\rm Sym}^{d_i} {\mb C}$ up to translation. We have the similar description if we add a marked point.
\end{example}

\section{Quantum Kirwan map}\label{section5}

\subsection{The quantum cohomology}

We assume that the symplectic quotient $\left( \ov{M}, \ov{\omega} \right)$ of $\left( M, \omega \right)$ is monotone, i.e., there exists a real number $c>0$ such that $c_1(T\ov{M}) = c [\ov{\omega}]$. We also assume that $H_2\left( \ov{M}; {\mb Z} \right)$ is torsion-free and there is an additive basis $A_1, \ldots, A_m$ of $H_2\left( \ov{M}; {\mb Z} \right) $ such that the homology class of any holomorphic sphere in $\ov{M}$ is of the form $d_1 A_1 + \cdots + d_m A_m$ with $d_i \geq 0$. Then we choose the Novikov ring $\Lambda$ to be the polynomial ring ${\mb R}[ q_1, \ldots, q_m]$ with ${\rm deg} q_i = 2c_1(A_i)$. The quantum cohomology ring of $\ov{M}$ is a ring with underlying abelian group
\begin{align}	
QH^*\left( \ov{M}; \Lambda \right) := H^* \left( \ov{M}; {\mb R} \right) \otimes_{\mb R} \Lambda.
\end{align}
Choose an additive basis $\{e_\nu\}$ of $H^* \left( \ov{M}; {\mb R} \right)$ over ${\mb R}$, the multiplication is defined, for every $a, b \in H^* \left( \ov{M}; {\mb R}\right)$, 	
\begin{align}
a *_q b :=  \sum_{A = d_1 A_1 + \cdots + d_m A_m } \sum_{\nu_1, \nu_2} GW_{A, 3}^{\ov{M}} ( a, b, e_{\nu_1} ) g^{{\nu}_1 {\nu}_2} e_{\nu_2} \otimes q_1^{d_1} \cdots q_m^{d_m}
\end{align}
and extended in a $\Lambda$-linear way to $H^* \left( \ov{M}; \Lambda \right)$. Here $\left(g^{\nu_1 \nu_2}\right)$ is the inverse matrix of the intersection matrix $\left( g_{\nu_1 \nu_2} := \int_{\ov{M}} e_{\nu_1} \cup e_{\nu_2} \right)$. This makes $QH^*\left( \ov{M}; \Lambda \right)$ a graded algebra over $\Lambda$, which is called the {\it quantum cohomology} of $\ov{M}$.

For example, for the case ${\mb P}^n$, $\Lambda= {\mb R}[q]$ with ${\rm deg} q = 2(n+1)$ and the quantum cohomology ring of ${\mb P}^n$ is
\begin{align}
QH^* \left( {\mb P}^n; \Lambda \right) = {\mb R}[c, q]/ \langle c^{n+1}= q\rangle.
\end{align}

\subsection{The Poincar\'e bundle and the evaluation maps}

Let's consider the moduli space of $(1, 1)$-marked stable affine vortices. We call the component which contains the interior marked point the {\it primary component}.

Fix a homology class $A\in H_2^G(M)$. Consider the category of $(1, 1)$-marked stable affine vortices, $\ov{\wt{\mc M}}_{1, 1}^{\mb A}(M, A)$ with homology class equal to $A$. The morphism set between two objects is the set of isomorphisms $\left( f, \{t_{\beta_k}, g_{\beta_k}\} \right)$ defined in Definition \ref{defn45}. Also consider the category of objects $\left( \wt{\bf W}, p \right)$ where $\wt{\bf W}$ is an object of $\ov{\wt{\mc M}}_{1, 1}^{\mb A}(M, A)$ and $p \in S^1 \subset {\mb C}$ which is regarded as a point in the fibre at the marked point $z_0$ of the trivial principal bundle of the primary component. A morphism between $\left( \wt{\bf W}, p \right)$ and $\left( \wt{\bf W}', p' \right)$ is an isomorphism $\left( f, \{ t_{\beta_k}, g_{\beta_k} \} \right)$ between $\wt{\bf W}$ and $\wt{\bf W}'$ such that $ p = p' g_{V_0}(z_0)$. Taking quotient modulo isomorphisms, we get a principal $G$-bundle over $\ov{\mc M}_{1, 1}^{\mb A}(M, A)$, denoted by 
\begin{align}
\ov{\mc P}_0 \to \ov{\mc M}_{1, 1}^{\mb A}(M, A).
\end{align}
The right $G$-action is induced from $\left( \wt{\bf W}, p \right)\cdot g =  \left( \wt{\bf W}, pg \right)$. There is a well-defined equivariant evaluation induced from $\left( \wt{W}, p \right) \mapsto u_{V_0}(p) \in M$, denoted by
\begin{align}
ev_0: \ov{\mc P}_0 \to M.
\end{align}
This induces a map, with an abuse of notation, $ev_0: \ov{\mc M}_{1, 1}^{\mb A}(M, A) \to M_G:= EG\times_G M$, to the Borel construction $M_G$ of $M$, whose cohomology is the equivariant cohomology of $M$. Also, we have the evaluation map $ev_\infty: \ov{\mc M}_{1, 1}^{\mb A}(M, A)\to \ov{M}$ by evaluating the root component at infinity. 

\subsection{The quantum Kirwan map}

Suppose we have a natural orientation of the deformation complex of the equation (\ref{vortexeqn}) (over ${\mb C}$) with gauge tranformations, and there is a well-defined virtual fundamental class $\left[ \ov{\mc M}^{\mb A}_{1, 1}(M, A ) \right]^{vir}$, then the quantum Kirwan map $\kappa_Q: H^*_G \left (M; {\mb R} \right) \to QH^* \left( \ov{M}; \Lambda \right)$ is defined by
\begin{align}\label{eqn57}
\kappa_Q(\alpha) = \sum_{A= d_1 A_1 + \cdots + d_m A_m}  \sum_{\nu_1, \nu_2} \left\langle (ev_0)^* \alpha \cup (ev_\infty)^* e_{\nu_1}, \left[ \ov{\mc M}^{\mb A}_{1,1}(M, A)\right]^{vir} \right\rangle \cdot g^{\nu_1 \nu_2} e_{\nu_2} \otimes q_1^{d_1} \cdots q_m^{d_m}.
\end{align}

Note that if we specialize $\kappa_Q(\alpha)$ at $q_1=\cdots = q_m = 0$, then the only contribution is given by $\ov{\mc M}_{1, 1}^{\mb A}(M, 0)$, which is homeomorphic to $\ov{M}$; the Poincar\'e bundle over this moduli is isomorphic to the	$G$-bundle $\mu^{-1}(0) \to \ov{M}$. So the result will be the classical Kirwan map $\kappa(\alpha)$.

\section{The adiabatic limit of $U(1)$-vortices}\label{section6}

\subsection{Holomorphic $N$-pairs and Hitchin-Kobayashi correspondence for rescaled area form}

From now on, we consider concrete examples. More precisely, we work with $M = {\mb C}^N$ with the diagonal $S^1$-action, whose moment map is
\begin{align}
\mu(z_1, \ldots, z_N) = -{ \sqrt{-1} \over 2} \left(\sum_{j=1}^N |z_j|^2-1 \right).
\end{align}
Here we identify ${\rm Lie} S^1 \simeq i {\mb R} \simeq \left({\rm Lie} S^1 \right)^*$. The symplectic quotient is the projective space ${\mb P}^{N-1}$. We assume $N \geq 2$.

Recall in Subsection \ref{subsection22}, a degree $d$, rank $1$, stable holomorphic $N$-pair over the Riemann surface $\Sigma$ is a tuple $\left( {\mc L}; \varphi_1, \ldots, \varphi_N \right)$ where ${\mc L}\to \Sigma$ is a degree $d$ holomorphic line bundle and $\varphi_j \in H^0({\mc L})$ such that at least one of them is nonzero. On the other hand, let $L\to \Sigma$ be a fixed smooth Hermitian line bundle of degree $d$; let $\Omega_\Sigma$ be a smooth area form on $\Sigma$. Let $\wt{\mc M}^\lambda_\Sigma \left( {\mb C}^N, L \right)$ be the space of all solutions $\left( A; \phi_1, \ldots, \phi_N \right)$ to the following equation:
\begin{align}\label{eqn62}
\left\{\begin{array}{ccl} \ov\partial_A \phi_j & = & 0;\\[0.2cm]
                           F_A & = & {\lambda^2 \sqrt{-1} \over 2} \left(  \sum_{j=1}^N \left| \phi_j \right|^2 - 1 \right) \Omega_\Sigma.
\end{array}\right.
\end{align}
Here $A$ is a unitary connection on $L$ and $\phi_i$ are smooth sections of $L$. And let ${\mc M}^\lambda_\Sigma\left( {\mb C}^N, L \right)$ be the space of gauge equivalence classes of such solutions.

By the Hitchin-Kobayashi correspondence (Theorem \ref{thm21}), the moduli space of (isomorphism classes of) degree $d$ rank $1$ stable holomorhpic $N$-pairs is homeomorphic to ${\mc M}^\lambda_\Sigma( {\mb C}^N, L)$ for any $\lambda$. In particular, we always regard a holomorphic line bundle ${\mc L}$ having the underlying smooth line bundle $L$ with a holomorphic structure given by $\ov\partial_A$, the $(0, 1)$-part of the connection $A$. Take $\lambda = 1$ and we represent a holomorphic $N$-pair to an element in $\wt{\mc M}^1_\Sigma ( {\mb C}^N, L)$. Then for each $\left( B; \varphi_1, \ldots, \varphi_N \right) \in \wt{\mc M}^1_\Sigma({\mb C}^N, L)$ and for each $\lambda>0$, there exists a unique $h_\lambda$ such that
\begin{align}\label{kazdanwarner}
 \Lambda F_B - \sqrt{-1} \Delta h_\lambda -   {\lambda^2 \sqrt{-1} \over 2} \left( e^{-2h_\lambda} \sum_{j=1}^N \left| \varphi_j \right|^2 -1 \right)=0.
\end{align}
The corresponding solution to (\ref{eqn62}) is given by 
\begin{align}
A = B + \sqrt{-1} d^c h_\lambda = B - \partial h_\lambda + \ov\partial h_\lambda,\  \phi_j = e^{-h_\lambda} \varphi_j
\end{align}
which is the result by applying the ``purely imaginary gauge tranformation'' $e^{ h_\lambda}$ to the tuple $\left( B; \varphi_1, \ldots, \varphi_N \right)$.

For any $\lambda$-vortex $ \left( A; \phi_1, \ldots, \phi_N \right) \in \wt{\mc M}^\lambda_\Sigma( {\mb C}^N, L)$, the $\lambda$-energy density and $\lambda$-energy are given by
\begin{align}
e_\lambda(z):=  \sum_{j=1}^N \left| d_A \phi_j (z) \right|^2 + {\lambda^2 \over 2} \left| \sum_{j=1}^N \left|\phi_j(z) \right|^2 -1 \right|^2 
\end{align}
\begin{align}\label{eqn66}
E_\lambda(A; \phi_1, \ldots, \phi_N): = \int_\Sigma e_\lambda(z) \Omega_\Sigma = 2\pi  {\rm deg} L.
\end{align}

Let $\left( {\mc L}, \ora{\varphi} \right) =\left( {\mc L};  \varphi_1, \ldots, \varphi_N\right)$ be a stable holomorphic $N$-pair. The {\it base locus} of $\ora{\varphi}$ is the intersection of the zeroes of $\varphi_j$, denoted by $Z\subset \Sigma$. The multiplicity $m_p$ of $p\in Z$ is the minimum of the multiplicities of $p$ as a zero of $\varphi_j$. Then $\ora{\varphi}$ defines a holomorphic map
\begin{align}
\begin{array}{cccc}
\left[ \ora{\varphi}\right]: &    \Sigma \setminus Z  & \to & {\mb P}^{N-1}\\
                         &  z & \mapsto & \left[ \varphi_1(z), \ldots, \varphi_N(z) \right]
                         \end{array}
\end{align}
which extends uniquely to a holomorhic map from $\Sigma$ to ${\mb P}^{N-1}$ by removal of singularity, which is still denoted by $\left[ \ora{\varphi}\right]$. It is easy to see that
\begin{align}
{\rm deg} \left[ \ora{\varphi} \right] = {\rm deg} L - \sum_{p\in Z} m_p.
\end{align}

\subsection{The adiabatic limit}

We will study the behavior of a sequence of objects $(A_\lambda, u_\lambda) \in \wt{\mc M}^\lambda_{\Sigma} ( {\mb C}^N, L)$ as $\lambda \to \infty$ and give a refined version of the bubbling zoology of the adiabatic limit. In particular, we will give an algebraic condition on the bubbling of nontrivial affine vortices.

From now on until the end of this section, we fix a sequence of stable holomorhpic $N$-pairs, which, by the Hitchin-Kobayashi correspondence, can be identified with a sequence $\left( B_k; \ora{\varphi}_k\right) = \left( B_k; \varphi_{1, k}, \ldots, \varphi_{N, k} \right) \in \wt{\mc M}^1_\Sigma\left( {\mb C}^N, L \right)$. We assume that this sequence converges to $\left( B; \varphi_1, \ldots, \varphi_N\right) \in \wt{\mc M}^1_\Sigma \left( {\mb C}^N, L \right)$, whose base locus is denoted by $Z\subset \Sigma$. We also fix a sequence $\lambda_k \to \infty$, hence for each $k$ let $h_k: \Sigma \to {\mb R}$ denote the unique solution to the equation (\ref{kazdanwarner}) for the vortex $\left( B_k; \ora{\varphi}_k \right)$ and $\lambda = \lambda_k$. We denote $A_k:= B_k - \partial h_k + \ov\partial h_k$ and $\phi_{j, k}:= e^{-h_k} \varphi_{j, k}$.

For each $p \in \Sigma$, we choose a local holomorphic coordiante $\xi_p: B_{r_p} \to U_p$, where $r_p>0$ and $B_{r_p} \subset {\mb C}$ is the radius $r_p$ disk centered at $0 \in {\mb C}$, $U_p$ is a neighborhood of $p\in \Sigma$. Up to rescaling, we can assume that $\xi_p^* \Omega_\Sigma= \sigma_p ds dt$, where $(s, t)$ is the standard coordinates on ${\mb C}$ and $\sigma_p: B_{r_p} \to {\mb R}_+$ is a smooth function, such that $\sigma_p(0) = 1$. We also trivialize $\xi_p^* L$ smoothly over $B_{r_p}$ by $\zeta_p: B_{r_p} \times {\mb C} \to \xi_p^* L$. We call such a triple $\left( U_p, B_{r_p}, \xi_p, \zeta_p \right)$ an admissible chart near $p$. Then any triple $\left( L , A, u\right)$ over $\Sigma$ can be pulled back by an admissible chart to $B_{r_p}$, to a triple $\left( B_{r_p} \times {\mb C},  d+ \alpha; \phi \right)$ over $B_{r_p}$.

Then for any $q \in \xi_p\left( B_{{r_p\over 2}} \right) \subset U_p$ and for $r\in \left( 0, {r_p \over 2}\right)$, the inclusion $B_r \simeq  B_{\xi_p^{-1}(q)}(r) \subset B_{r_p}$ induces a chart near $q$, from the admissible chart $(U_p, B_{r_p}, \xi_p, \zeta_p)$. Then for any large number $\lambda$, we zoom in by $B_{\lambda r} \simeq B_r$. With the admissible chart near $p$ understood, we abbreviate by
$$s_{q, \lambda}^* \left(A, u \right)$$
to be the pull-back object on the trivial bundle over $B_{\lambda r}$.

\begin{lemma}\label{lemma41} 
Suppose $p_k \in \Sigma$ is a sequence of points such that $\lim_{k \to \infty} p_k = p$ and $r_k>0$ is a sequence of real numbers such that \begin{align}\label{410}
\lim_{k \to \infty} r_k = 0,\  \lim_{k \to \infty} R_k: = \lim_{k \to \infty} \lambda_k r_k = + \infty.
\end{align}
Choosing an adimissible chart centered at $p$. Then we obtain a sequence of pull-back pairs $$\left( d+ \alpha_k; \phi_k \right):= s_{p_k, \lambda_k}^*\left( A_k, u_k \right)$$ on $B_{R_k}$. Then, there exists a subsequence (still indexed by $k$), such that $\left( d+ \alpha_k, \phi_k\right)$ converges in $C^\infty_{loc}$ on ${\mb C}$ to an affine vortex with finite energy. Note that the limit can be trivial, and we don't need to take gauge transformations.
\end{lemma}

\begin{proof}
Let's first see why we don't need to take gauge transformations. Indeed, with respect to the admissible chart write
\begin{align}
B_k = d + \beta_k = d+  \Phi_k dx + \Psi_k dy.
\end{align}
Abbreviate $s_k^* = s_{p_k, \lambda_k}^*$, we have
\begin{align}
s_k^* \beta_k = \lambda_k^{-1} \left( s_k^* \Phi_k ds + s_k^* \Psi_k dt \right)
\end{align}
converges to zero in any compact subset $K\subset {\mb C}$, because the sequence of connections $B_k$ converges on $\Sigma$. Now we see that
\begin{align}
s_k^* A_k = d + s_k^* \beta_k - \sqrt{-1} s_k^* d^c h_k
\end{align}
with $d( s_k^* \beta_k - \sqrt{-1} s_k^* d^c h_k)$ uniformly bounded, $d^* ( s_k^* \beta_k - \sqrt{-1} s_k^* d^c h_k) \to 0$. Hence there exists a subsequence of $s_k^* A_k$ (still indexed by $k$) converging in $C_{loc}^0$.	

Then, on the product $B_{R_k} \times {\mb C}^N$, the connection $s_k^*A_k$ and the almost complex structures on ${\mb C}^N$ and $B_{R_k}$ induces a sequence of almost complex structures $J_k$ with respect to which $\ora\phi_k$ is holomorphic. $J_k$ converges (in $C^0$) because $s_k^* A_k$ converges weakly in $W^{1, p}(K)$. The energy density of $\ora{\phi}_k$ is also uniformly bounded by (\ref{eqn314}) and (\ref{eqn66}). By the standard method, a subsequence of $\ora{\phi}_k$ converges to a section $\ora{\phi}_\infty$ which is holomorphic with respect to the connection $d+ \alpha_\infty$ on ${\mb C}$. And it is easy to check that the limit pair $\left( d+\alpha_\infty, \ora{\phi}_\infty \right)$ satisfies the vortex equation on ${\mb C}$ with respect to the area form $ds dt$. Finally it is standard to show that the limit is actually smooth and the convergence is in $C^\infty_{loc}$.
\end{proof}

The above lemma is somehow a fact {\it a priori}, which will be used for many times to guarantee the existence of converging subsequence.

\subsection{Convergence away from base locus}

\begin{prop} If the limit of the sequence of points $p_k$ in Lemma \ref{lemma41} lies in $\Sigma \setminus Z$, then for any sequence $r_k>0$ satisfying  (\ref{410}), the limit affine vortex of any convergent subsequence we obtained in Lemma \ref{lemma41} is trivial. In particular, the function $\mu_k = \sum_{j=1}^N \left| \phi_{j, k} \right|^2 - 1$ converges to zero in $C_{loc}^0 \left( \Sigma \setminus Z \right)$.
\end{prop}

\begin{proof}
In Lemma \ref{lemma41}, denote by $h_k'= s_k^* h_k$, then (\ref{kazdanwarner}) implies
\begin{align}
\Delta h_k'  + {1\over 2} \left( e^{- 2 h_k'} \sum_{j=1}^N \left| s_k^* \psi_{j,k} \right|^2 -1 \right) + {\sqrt{-1} \over  \lambda_k^2 } s_k^*\left( \Lambda F_{B_k}\right)=0	.
\end{align}	
Because $p$ is not in the base locus, the sequence of functions $\sum_{j=1}^N \left| s_k^* \psi_{j,k} \right|^2$ converges to a nonzero constant $b$ uniformly on any compact subset of ${\mb C}$. Then the sequence $h_k'$ converges to a solution to the following Kazdan-Warner equation on ${\mb C}$
\begin{align}
\Delta h + {1\over 2} \left( b e^{- 2h } - 1 \right) = 0.
\end{align}
But there is only one solution $h$ which has assymptotic value ${1\over 2} \log b$, which is the constant. This implies that the vortex is trivial. 

Then if there exists $p_k \in K$ such that $\lim_{k \to \infty} p_k = p \in K$ and $\lim_{k \to \infty} \left| \mu_k(p_k) \right| >0$, this means the energy density blows up like $\lambda_i^{-2}$ near $p$. Then as in \cite{Gaio_Salamon_2005}, the process in Lemma \ref{lemma41} will produce a nontrivial affine vortex, which contradicts with the above.
\end{proof}

Then away from the base locus, the sequence of sections $\ora{\phi}_k$ will sink into the level set $\mu^{-1}(0)$. Then we expect that, by projecting to the quotient $\mu^{-1}(0)/ S^1$, it will converges to a holomorphic map in ${\mb P}^{N-1}$. The precise meaning is described as follows.

For $\ora{z} \in {\mb C}^N$ with $|\mu(\ora{z})| < {1\over 2}$, we have the map
\begin{align}
\begin{array}{ccc}
\mu^{-1}\left( \sqrt{-1} (-{1\over 2}, {1\over 2}) \right) & \to & \mu^{-1}(0) \times \sqrt{-1} (-{1\over 2}, {1\over 2})\\
                        \ora{z} & \mapsto & \left( \left( \sum_{j=1}^N |z_j|^2 \right)^{-{1\over 2}} \ora{z} , \mu(\ora{z}) \right).
\end{array}
\end{align}
This induces the map $\pi: \mu^{-1}\left( \sqrt{-1} (-{1\over 2}, {1\over 2}) \right) \to {\mb P}^{N-1}$ by projecting.

\begin{prop}
For any compact subset $K \subset \Sigma \setminus Z$, the sequence of maps $\pi\circ \ora{\phi}_k: K \to {\mb P}^{N-1}$ converges to the holomorphic map $\left[ \ora{\psi} \right]|_K: K \to {\mb P}^{N-1}$.
\end{prop}
\begin{proof} Indeed, observe that $\left(\pi\circ \ora{ \phi}_k \right)|_K = \left[ \ora{\psi}_k  \right]|_K$. And since $\ora{\psi}_k $ converges to an $N$-pair which is base point free over $K$, the convergence is obvious.
\end{proof}

\subsection{Bubbling at base points}

Now for each $p \in Z$, take an admissible chart $(U_p, B_r, \xi_p, \zeta_p)$ near $p$. Also for each $k$, take a local holomorphic section $e_k$ such that $e_k(0) = \zeta_p(0)(1) \in L_p$ and assume that $\lim_{k \to \infty} e_k$ converges to a smooth section of $L$ over $U_p$.

With respect to the admissible chart, 
\begin{align}
B_k  = d + \beta_k ,\ \beta_k \in \Omega^1 \left( U, i {\mb R}\right),\  B = d + \beta
\end{align}
and $\beta_k \to \beta$, $\psi_{j, k} \to \psi_j$ in $C^\infty$-topology.

For each $k$, and $j= 1, \ldots, N$, denote by $Z_{j, k} = \left( \phi_{j, k} \right)^{-1}(0) \cap U$. Each element $z_{j, k} \in Z_{j, k} $ has an associated multiplicity ${\mf m}(z_{j, k}) \in {\mb Z}^+$. By taking a subsequence if necessary, we assume that there are finite sets ${\mf Z}_j$ and bijections
$$\rho_k : {\mf Z}_j \to Z_{j, k}$$
and maps ${\mf m}: {\mf Z}_j \to {\mb Z}_{>0}$ such that ${\mf m} \left( \rho_k  ({\mf z}_j ) \right) ={\mf m} ({\mf z}_j)$. (Note that $Z_{j, k}$ could be empty. This happens only if the section $\phi_{j, k}$ converges to zero on $\Sigma$.)

We then consider all sequences $\{p_k \}\subset U$ that converges to $p$, which are identified with a sequence of complex numbers $\left\{ z_k  \right\}$ converging to the origin. By taking a further subsequence if necessary, we assume that for each $j$, there is a subset ${\mf W}_j\subset {\mf Z}_j$ (associated to the sequence $\left\{ z_k  \right\}$) such that
\begin{align}
\limsup_{k \to \infty} \lambda_k \left| z_k-  \rho_k ( {\mf w}_j) \right| < +\infty, \forall {\mf w}_j \in {\mf W}_j;
\end{align}
\begin{align}
\liminf_{k \to \infty} \lambda_k  \left| z_k - \rho_k  ( {\mf z}_j )  \right| = +\infty, \forall {\mf z}_j \in {\mf Z}_j \setminus {\mf W}_j.
\end{align}

We can write that for $z \in U$,
\begin{align}
\phi_{j, k} (z) = e_k (z)  \prod_{ {\mf z}_j \in {\mf Z}_j } \left( z- \rho_k ({\mf z}_j) \right)^{{\mf m}({\mf z}_j)} f_{j, k} (z)
\end{align}
where $f_{j,k}$ is a nonvanishing holomorphic function. Then the assumption that $\psi_{j, k}$ converges to $\psi_j $ implies that
\begin{align}
f_j := \lim_{\nu \to \infty} f_{j, k}
\end{align}
exists.

For each $j$, denote
\begin{align}
d_j:= \sum_{{\mf w}_j \in {\mf W}_j} {\mf m}( {\mf w}_j) \geq 0.
\end{align}

Define
\begin{align}
t_k := \max\left\{ \lambda_k^{- d_j} \left| f_{j, k} \left( z_k \right) \right| \prod_{{\mf z}_j \in {\mf Z}_j \setminus {\mf W}_j} \left| z_k  - \rho_k  ({\mf z}_j)\right|^{{\mf m}({\mf z}_j)} : {\mf W}_j \neq \emptyset \right\}
\end{align}
\begin{align}
T_k := \max \left\{  \left| f_{j, k} \left( z_k \right) \right| \prod_{{\mf z}_j \in {\mf Z}_j} \left| z_k - \rho_k  ({\mf z}_j) \right|^{{\mf m}({\mf z}_j)} : {\mf W}_j = \emptyset\right\}
\end{align}
with the convention that $\max \emptyset = 0$.

\begin{lemma}\label{lemma64}
If $\liminf_{k \to \infty} {T_k  \over t_k } < +\infty$, then a subsequence of the sequence $s_k^*\left( A_k; \ora{\phi}_k  \right)$ converges to a nontrivial affine vortex; if $\liminf_{k \to \infty} {T_k \over t_k } = +\infty$, then there is a subsequence $s_k^* \left( A_k; \ora{\phi}_k \right)$ converges to a trivial affine vortex.
\end{lemma}

\begin{proof} 
In the first case, we assume by taking a subsequence that there exists $j_0$ such that for all $k$,
$$t_k  = \lambda_k^{- d_{j_0}} \left| f_{j_0,k } \left( z_k \right) \right| \prod_{{\mf z}_{j_0} \in {\mf Z}_{j_0} \setminus {\mf W}_{j_0}} \left| z_k - \rho_k  ({\mf z}_{j_0}) \right|^{{\mf m}( {\mf z}_{j_0})}.$$
And, if $\lim_{k \to \infty} \psi_{j_0, k} = 0$, we may assume by taking a subsequence that 
\begin{align}
\lim_{k \to \infty} {f_{j_0, k} \over \left| f_{j_0, k }(z_k ) \right| } 
\end{align}
converges to a nonvanishing holomorphic function on $U$.

Then we have
\begin{multline}
s_k^* \phi_{j_0, k}(w) = e^\nu(z) e^{-h_k\left( z_k + \lambda_k^{-1}w \right)} \prod_{{\mf z}_{j_0} \in {\mf Z}_{j_0}} \left( {w\over \lambda_k} + \left( z_k - \rho_k ({\mf z}_{j_0}) \right) \right)^{{\mf m}({\mf z}_{j_0})} f_{j_0, k} ( z_k + \lambda_k^{-1} w) \\
=  e_k(z) e^{-h_k \left( z_k + \lambda_k^{-1}w \right)} \prod_{{\mf w}_{j_0} \in {\mf W}_{j_0}} \left( {w\over \lambda_k} + \left( z_k - \rho_k ( {\mf w}_{j_0}) \right) \right)^{ {\mf m}( {\mf w}_{j_0})} \prod_{ {\mf z}_{j_0} \in {\mf Z}_{j_0} \setminus  {\mf W}_{j_0}} \left( {w\over \lambda_k} + \left( z_k - \rho_k ({\mf z}_{j_0}) \right) \right)^{{\mf m}( {\mf z}_{j_0})}  f_{j_0, k}\\
= e_k(z) e^{-h_k \left( z_k + \lambda_k^{-1}w \right)}  t_k \prod_{{\mf w}_{j_0} \in {\mf W}_{j_0}} \left( w +  \lambda_k \left( z_k - \rho_k ({\mf w}_{j_0}) \right) \right)^{{\mf m}({\mf w}_{j_0})} \prod_{ {\mf z}_{j_0} \in {\mf Z}_{j_0} \setminus  {\mf W}_{j_0}} \left( {w\over \lambda_k ( z_k - \rho^\nu ({\mf z}_{j_0}))} + 1 \right)^{{\mf m}({\mf z}_{j_0})} \\
\cdot  {f_{j_0, k}(z_k + \lambda_k^{-1}w) \over \left| f_{j_0, k}(z_k)\right|} \prod_{{\mf z}_{j_0} \in {\mf Z}_{j_0} \setminus {\mf W}_{j_0}} \left( { z_k - \rho_k ( {\mf z}_{j_0}) \over \left| z_k - \rho_k( {\mf z}_{j_0}) \right|} \right)^{ {\mf m}( {\mf z}_{j_0})}.
\end{multline}
Then by taking a subsequence, the product of the last four factors converges to a nonzero polynomial, uniformly on any compact subset of ${\mb C}$, which has zeroes at $\lim_{\nu \to \infty}\lambda_\nu \left( z^\nu- \rho^\nu( {\mf w}_{j_0}) \right)$ for each ${\mf w}_{j_0} \in {\mf W}_{j_0}$. Then by the {\it a priori} convergence of $s_k^* \phi_{j_0, k}$ (Lemma \ref{lemma41}), this implies that the function
\begin{align}
e^{-H_k(w)}: = e^{- h_k \left( z_k + \lambda_k^{-1} w \right) + \log t_k}
\end{align}
converges to a smooth function, denoted by $H_0$. Then, for all $j\neq j_0$, 
\begin{multline}
s_k^* \phi_{j, k}(w) = e_k(z) e^{-h_k\left( z_k + \lambda_k^{-1}w \right)} \prod_{{\mf z}_j \in {\mf Z}_j} \left( {w\over \lambda_k} + (z_k - \rho_k( {\mf z}_{j})) \right)^{{\mf m} ( {\mf z}_{j})} f_{j ,k}\\
= e_k(z) f_{j, k} e^{-H_k(w)} t_k^{-1} \prod_{ {\mf w}_{j} \in {\mf W}_{j}} \left( {w\over \lambda_k} + (z_k -\rho_k( {\mf w}_{j}) ) \right)^{{\mf m}( {\mf w}_{j})} \prod_{ {\mf z}_{j} \in {\mf Z}_{j} \setminus {\mf W}_j} \left( {w\over \lambda_k} + (z_k - \rho_k ({\mf z}_{j})) \right)^{{\mf m}({\mf z}_{j})}\\
= e_k(z) e^{-H_k(w)} \prod_{ {\mf w}_j \in {\mf W}_j } ( w + \lambda_k( z_k- \rho_k({\mf w}_j)) )^{{\mf m}({\mf w}_j)} \prod_{ {\mf z}_{j} \in {\mf Z}_{j} \setminus  {\mf W}_j} \left( {w\over \lambda_k ( z_k - \rho_k( {\mf z}_j) )} + 1 \right)^{{\mf m}({\mf z}_j)} { f_{j, k}(z_k + \lambda_k^{-1}w)  \over |f_{j, k}(z_k)|}
\\
\cdot \prod_{ {\mf z}_j \in {\mf Z}_j \setminus {\mf W}_j} \left( z_k - \rho_k( {\mf z}_j ) \right)^{{\mf m}({\mf z}_j)} { \left| f_{j, k}(z_k)\right| \over \lambda_k^{d_j} t_k}.
\end{multline}
Then taking a subsequence, we can assume that
\begin{align}
a_j:= \lim_{k \to \infty} \prod_{ {\mf z}_j \in {\mf Z}_j \setminus {\mf W}_j} \left( z_k - \rho_k( {\mf z}_j ) \right)^{{\mf m}({\mf z}_j)} { \left| f_{j, k}(z_k)\right| \over \lambda_k^{d_j} t_k}\in {\mb C}
\end{align}
exists. Then taking a further subsequence, we have
\begin{align}
\lim_{k \to \infty} s_k^* \phi_{j, k}(w) = e_\infty(0) H_0(w) a_j \prod_{{\mf w}_j \in {\mf W}_j} \left(  w+ \lim_{k \to \infty} \lambda_k( z_k - \rho_k({\mf w}_j) )\right)^{{\mf m}({\mf w}_j)} .
\end{align}
Since we know that the limit section $\lim_{k \to \infty} s_k^* \left( \phi_{1, k}, \ldots, \phi_{N, k} \right)$ doesn't vanish identically, $H_0 \neq 0$ and hence the sequence of functions $\wt{h}_k:= s_k^* h_k - \log t_k$ converges to a smooth function $\wt{h}$ on ${\mb C}$. And the limit vortex is nontrivial because at least for $j = j_0$, it is of the form 
$$a_{j_0} e^{- \wt{h}(w)} \prod_{s=1}^{d_{j_0}} ( w+ \xi_s)$$

Now we look at the second case: $\liminf_{k \to \infty} {T_k \over t_k} = +\infty$. Then we may assume that there exists a subsequence and $j_0$ such that
\begin{enumerate}
\item For each $j$, $\lim_{k \to \infty} \psi_{j, k} = 0 \Longrightarrow { f_{j, k} \over \left|f_{j, k}(z_k)\right|}$ converges to a nonvanishing holomorphic function on $U$;

\item ${\mf W}_{j_0} = \emptyset$ and $$T_k= \left| f_{j_0, k} (z_k) \right| \prod_{{\mf z}_{j_0} \in {\mf Z}_{j_0}} \left| z_k - \rho_k ( {\mf z}_{j_0}) \right|^{{\mf m}({\mf z}_{j_0})}.$$
\end{enumerate}
Then 
\begin{multline}
s_k^* \phi_{j_0, k}(w) = e_k(z) e^{- h_k(z)} f_{j_0, k} \prod_{ {\mf z}_{j_0} \in {\mf Z}_{j_0}} \left( {w\over \lambda_k} + ( z_k - \rho( {\mf z}_{j_0})) \right)^{ {\mf m}( {\mf z}_{j_0})}\\
= e_k(z) e^{- h_k(z)} T_k \prod_{ {\mf z}_{j_0} \in {\mf Z}_{j_0}} \left( { w\over \lambda_k ( z_k - \rho_k ( {\mf z}_{j_0}))} + 1 \right)^{ {\mf m}( {\mf z}_{j_0})}  { f_{j_0, k} \over \left| f_{j_0, k}(z_k) \right|} \prod_{{\mf z}_{j_0} \in {\mf Z}_{j_0}} \left( { z_k - \rho_k( {\mf z}_{j_0}) \over \left| z_k - \rho_k ( {\mf z}_{j_0}) \right| }\right)^{{\mf m}( {\mf z}_{j_0})}.
\end{multline}
(By taking a subsequence) the product of the last three factors converges to a nonzero constant, which implies that $e^{-h_k(z_k + \lambda_k^{-1} w)} T_k$ converges to a smooth function $H_0: {\mb C}\to {\mb R}$.

Then by our assumption $\lim_{k \to \infty} {T_k\over t_k} = \infty$, we can easily see that $s_k^* \phi_{j_0, k}$ dominates other $s_k^* \phi_{j, k}$. In particular, if ${\mf W}_j\neq \emptyset$, then 
\begin{align}
\lim_{k \to \infty} s_k^* \phi_{j, k} = 0;
\end{align}
if ${\mf W}_j = \emptyset$, then 
\begin{align}
\lim_{k \to \infty} s_k^* \phi_{j, k} = a_j H_0
\end{align}
for some constant $a_j$. Since the limit of $s_k^* \left( \phi_{1,k}, \ldots, \phi_{N, k} \right)$ doesn't vanish identically, we see that the sequence of functions $h_k( z_k + \lambda_k^{-1} w) - \log T_k$ converges to a smooth function $\wt{h}$ on ${\mb C}$, which must be (by Lemma \ref{lemma41}) the solution to the Kazdan-Warner equation
$$\Delta \wt{h} + {1\over 2} \left( e^{- 2{\wt{h}}} \sum_{{\mf W}_j= \emptyset} |a_j|^2 - 1  \right) =0 .$$
This implies that the limit vortex is trivial.
\end{proof}

From the proof we also see, if we replace $z_k$ be $z_k' $ with $\limsup_{k \to \infty} \lambda_k^{-1} |z_k - z_k'| < \infty$, then the nontrivial affine vortex we get from the first case will differ by a translation. 

\begin{rem} If we assume that $\varphi_j = \lim_{k \to \infty} \varphi_{j, k}$ is nonzero for each $j$, then from the proof of the above lemma, we see that a nontrivial affine vortex bubbles off if each $\phi_{j, k}$ contributes at least one zero so that they concentrate in a rate no slower than $\lambda_k\to \infty$ (the case where some $\varphi_j=0$ is more subtle). Due to the energy quantization (i.e., the energy of nontrivial affine vortices is bounded from below), we can find all possible nontrivial affine vortex bubbles at the base point $p$. There might be some missing degrees, caused by the zeroes of $\phi_{j, k}$ which concentrate in a slower rate than $\lambda_k^{-1}$. This slower concentration means holomorphic spheres in ${\mb P}^{N-1}$ bubble off, and the bubbling happens at the ``neck region'' between different affine vortices and between affine vortices and the original domain $\Sigma$. 
\end{rem}

\section{Classification of $U(1)$-vortices and its moduli spaces}\label{section7}
	
We consider an arbitray affine vortex $\left( A,  \ora{\phi} \right)$ with target ${\mb C}^N$ with finite energy. 

\begin{lemma}
There exists a complex gauge transformation $g= e^{ h_1 + i h_2}$ such that $g^* A$ is the trivial connection. If $g'$ is another complex gauge tranformation which also transform $A$ to the trivial connection, then $g' = g e^f$ where $f$ is an entire function.
\end{lemma}

\begin{proof}
Suppose $A= d + \Phi ds + \Psi dt$. We first show that there exists a unitary gauge transformation $g_1 = e^{i f}$ such that
\begin{align}\label{eqn61}
d^*\left( g_1^*A - d\right) = 0,
\end{align}
i.e., the Coulomb gauge condition. Indeed, 
\begin{align}
d^* \left( g_1^*A - d \right) = d^* \left( \left( \Phi - i {\partial f\over \partial x} \right) dx + \left( \Psi - i {\partial f \over \partial y}\right) dy    \right)=  i \Delta f - {\partial \Phi\over \partial x} - {\partial \Psi \over \partial y}.
\end{align}
Hence there exists $f$ such that (\ref{eqn61}) holds. Then for $g_2 = e^h$ with $h$ a real valued function,
\begin{align}
g_2^* g_1^* A - d = -i \left( {\partial h \over \partial x} dy - {\partial h \over \partial y} dx \right) + (g_1^*A  -d) = - i * dh + ( g_1^* A - d) = -i * \left(    dh  - i * ( g_1^*A -d)  \right).
\end{align}
The existence of $h$ such that $g_2^* g_1^* A = d$ follows from Poincar\'e lemma.
\end{proof}

Then, up to a unitary gauge transformation, we can assume that $\left( A, \ora{\phi} \right)$ is in Coulomb gauge, and $A = d - \partial h + \ov\partial h$ for some real valued function $h$ and $\ov\partial_A = \ov\partial + \ov\partial h = e^{-h} \ov\partial e^h$. So $\ov\partial_A \phi_j = 0 \Leftrightarrow \ov\partial ( e^h \phi_j) = 0$. The vortex equation is equivalent the following equation on real valued function $h: {\mb C}\to {\mb R}$
\begin{align}\label{eqn74}
\Delta h + {1\over 2} \left( e^{-2h} \sum_{j=1}^N \left| \psi_j(z) \right|^2 - 1 \right) = 0
\end{align}
where $\psi_j$ are entire functions. 

\begin{lemma}
If $E\left( A, \ora{\phi} \right) < \infty$, then $\psi_j = e^h \phi_j$ are polynomials and the Maslov index of $\left( A, \ora{\phi} \right)$ is equal to the maximum of the degrees of $\psi_j$. 
\end{lemma}

\begin{proof}
By Proposition \ref{prop39} there exists $g: S^1 \to U(1)$ such that for each $j$ and all $\theta \in S^1$, 
\begin{align}\label{eqn75}
\lim_{r\to +\infty} g(e^{-i \theta}) e^{- h ( r e^{i \theta})} \psi_j(r e^{i \theta}) \to a_j
\end{align}
and $\sum_{j=1}^N  |a_j|^2 = 1$. The Maslov index of the vortex is by definition the degree of $g$. 

Then for each $j$ with $a_j \neq 0$, (\ref{eqn75}) implies that $\infty$ is not an essential singularity of $\psi_j$. Hence $\psi_j$ is a polynomial. Then it is easy to see that $a_j \neq 0 \Longrightarrow {\rm deg} \psi_j = d$. Moreover, if $a_l = 0$ and $a_j \neq 0$, then
\begin{align}
\lim_{r\to \infty} {\psi_l (r e^{i \theta}) \over \psi_{j} ( r e^{i \theta}) } = {a_l \over a_j } = 0
\end{align}
which implies that $\psi_l$ is a polynomial with degree strictly less than $d$. 
\end{proof}

Before we proceed, we consider the uniqueness of the solutions to the Kazdan-Warner equation (\ref{eqn74}). For given $N$ polynomials $\left\{ \psi_j\right\}_{1\leq j \leq N}$, suppose $\sum_{j=1}^N \left|\psi_j(z) \right|^2$ are assymptotic to $c^2 |z|^{2d}$ as $|z|\to \infty$ for some $c>0$. 

\begin{lemma}\label{lemma73}
If $h_1, h_2: {\mb C}\to {\mb R}$ are both smooth and solve equation (\ref{eqn74}) on ${\mb C}$ such that $$\lim_{|z|\to \infty} \left( e^{-2h_i(z)} \sum_{j=1}^N	 \left| \psi_j(z) \right|^2 - 1\right) |z|^{4-\epsilon} = 0,\ i=1,2,$$
then $h_1 = h_2$.
\end{lemma}

\begin{proof}
We have 
\begin{align}
\Delta (h_1 - h_2) +   \left( e^{-2h_1} - e^{-2 h_2} \right) \sum_{j=1}^N \left| \psi_j\right|^2 =0.
\end{align}
Then the $L^2$-pairing with $h_1-h_2$ gives
\begin{align}
0 = \left| d(h_1- h_2) \right|_{L^2}^2 +  \int_{\mb C}  (h_1 - h_2) \left( e^{-2h_2} - e^{-2 h_1} \right)   \sum_{j =1}^N \left| \psi_j\right|^2 \geq 0.
\end{align}
So this is true only if $\psi_j \equiv 0$ or $h_1 = h_2$.
\end{proof}

\subsection{Construction of arbitrary vortices}

We have seen that all finite energy affine vortices are obtained from the solution to (\ref{eqn74}) for $\psi_j$ polynomials, and the solution, if exists, is unique. Now with our understanding of the adiabatic limit of equation (\ref{vortexeqn}) and the bubbling off phenomenon at the base locus (in the proof of Lemma \ref{lemma64}), it is rather easy to construct solutions to (\ref{eqn74}) with any given set of polynomials.

Indeed, suppose we are given $N$ polynomials $\psi_1, \ldots, \psi_N$, with $\psi_i= a_i(z- z_i^1)\cdots (z- z_i^{d_i})$. It is trivial to add zero polynomials hence we assume that $a_i \neq 0$. Let $d:= \max_{1\leq i \leq N} d_i$.

Now, consider the degree $d$ line bundle ${\mc O}(d) \to {\mb P}^1$. Take the sequence $\lambda_k = k$. Consider the sequence of sections of ${\mc O}(d)$, where are sequences of degree $d$ polynomials
\begin{align}
\varphi_i^{(k)}(z) =  k^{d_i - d}  a_i(1-z)^{d- d_i} \prod_{j=1}^{d_i} \left( z-  {z_i^j\over k} \right). 
\end{align}
Then for each $k$, $\left( {\mc O}(d); \varphi_1^{(k)}, \ldots, \varphi_N^{(k)}\right) $ is a rank 1 stable holomorphic $N$-pair over ${\mb P}^1$. By the Hitchin-Kobayashi correspondence, there exists a metric $H^{(k)}$ on ${\mc O}(d)$ which solves the vortex equation. Now the origin is a base point of the limit $N$-pairs $\left( {\mc O}(d); \phi_1, \ldots, \phi_N \right)$ with $\phi_i(z) = a_i z^d$ for $d_i = d$ and $\phi_i(z) = 0$ for $d_i < d$. As in the proof of Lemma \ref{lemma64}, we take $p= p_k = 0$ and $r_k:= \lambda_k^{-{1\over 2}} = k^{-{1\over 2}}$. It is easy to see that the sequence satisfies the criterion $\lim_{k \to \infty} {T_k \over t_k} < \infty$ of Lemma \ref{lemma64}. Hence by zooming in with a factor $\lambda_k = k$, we see a nontrivial affine vortex of Maslov index $d$ bubbles off. Equivalently, we constructed the unique solution to the Kazdan-Warner equation (\ref{eqn74}) with the given polynomials $\psi_j$ on ${\mb C}$.

\subsection{Identify the moduli space}

Denote
\begin{align}
\wt{N}_d:= \left\{ \left( a_{jl} \right)_{1\leq j \leq N,\ 0 \leq l \leq d} \in {\mb C}^{(d+1)N}\ |\ ( a_{1d}, \ldots, a_{Nd} ) \neq (0, \ldots, 0)\right\}.
\end{align}
This corresponds to the space of $N$ polynomials as the input of (\ref{eqn74}). ${\mb C}^*$ acts on $\wt{N}_d$ freely and denote $N_d:= \wt{N}_d/ {\mb C}^*$. We will use ${\bf x}$ to denote a general point in $N_d$. 

\begin{cor}\label{cor63}
There is an orientation preserving homeomorphism \begin{align}
\Phi_{HK}^d: N_d \to  {\mc M}_{1, 1}^{\mb A} \left({\mb C}^N ,d \right). 
\end{align}
And for each $(A, u) \in \wt{\mc M}_{1, 1}^{\mb A}\left({\mb C}^N, d\right)$, the linearization $D_{A, u}$ defined in (\ref{eqn39}) is surjective.
\end{cor}

\begin{proof} We define $\wt{\Phi}^d_{HK}: N_d \to \wt{\mc M}_{0, 1}^{\mb A}\left({\mb C}^N, d \right)$ to be the map which assigns to ${\bf x}= \left[ a_{jl} \right]_{1\leq j \leq N, 0\leq l \leq d}$ the affine vortex
$$\left( A = d -\partial h +\ov\partial h, u = e^{-h} \left( \psi_1, \ldots, \psi_N \right)  \right)$$
where $\psi_j(z) = \sum_{l=0}^d a_{jl} z^l$ and $h$ is the solution to (\ref{eqn74}) with the input $\psi_1, \ldots, \psi_N$. Then define $
\Phi_{HK}^d({\bf x})$ to be the equivalence class of $(1, 1)$-marked affine vortex $\left( \wt{\Phi}_{HK}^d({\bf x}), 0 \right)$, where $0\in {\mb C}$ is the interior marked point.

To show show that map is continuous, we need to show that the solution to the Kazdan-Warner equation depends continuously on the polynomials $\psi_1, \ldots, \psi_N$, which can be proved in a standard way. The orientation-preserving property then follows from a $C^1$-dependence on $\psi_1, \ldots, \psi_N$.

Now we show the surjectivity of the linearization map. First notice that the linearization of gauge tranformation is injective, because the image of $u$ is not in the fixed point set of the $S^1$-action on ${\mb C}^N$. Hence it suffices to prove that the map
\begin{align}
D_{A, u} \left( V, \alpha \right) = \left( \begin{array}{c} \ov\partial_A V + {\mc X}_\alpha^{0, 1}(u)\\
                                                            d\alpha +  d\mu (V) ds dt                  \end{array}\right)
\end{align}
is surjective, where $(V, \alpha) \in W_{p, \delta}$ which is defined in Subsection \ref{subsection31}. Note that the space $W^{1, p}_\delta\left( {\mb C}^N \right) := \left\{ V \in W^{1, p}_{loc}\left( {\mb C}^N \right) \ |\  \left| V \right|_{L_\delta^p} + \left| \nabla^A V \right|_{L_\delta^p} < \infty \right\}$ and $W^{1, p}_\delta \left( \Lambda^{1, 0}{\mb C}\right)$ are contained in $W_{p, \delta}$ by obvious inclusions. Hence it suffices to show that the two maps
\begin{align}\label{611}
\begin{array}{cccc}
\ov\partial_A :&  W^{1, p}_\delta \left( {\mb C}^N \right)   & \to  &  L_p^\delta \left( \Lambda^{0, 1}{\mb C} \otimes {\mb C}^N \right)\\
               &   	V & \mapsto & \ov\partial_A V
               \end{array}
\end{align}
\begin{align}\label{612}
\begin{array}{cccc}
\ov\partial: & W^{1, p}_\delta\left( \Lambda^{1,0}{\mb C} \right)&  \to &  L^p_\delta \left( \Lambda^2 {\mb C} \right)\\
             &    \alpha^{1, 0} & \mapsto & \ov\partial \alpha^{1, 0}
             \end{array}
\end{align}
both have dense range (they are not Fredholm).

To prove that (\ref{611}) has dense range, it suffices to prove for $N= 1$. Then the operator is
\begin{align}
\ov\partial_A f = \left( e^{- h} \ov\partial ( e^h f) \right). 
\end{align}
If its range is not dense, then there exists $g\in \left( L_\delta^p \right)^* = L_{-\delta}^q$ such that
\begin{align}
\int_{\mb C} g dz \wedge \ov\partial_A f = 0,\ \forall f \in C^\infty_c. 
\end{align}
This implies that $e^{-h} g$ is an entire function. But $e^h$ grows like $|z|^d$ for $d\geq 0$. Hence $g= 0$ or $g$ grows at least like $|z|^d$. For such a $g$ in the latter case to lie in $L_{-\delta}^q$, we should have
\begin{align}
-2 > q ( d- \delta) \Longrightarrow \delta > d + 2 - {2\over p} \geq 2 - {2\over p}
\end{align}
which contradicts with our choice of $p$ and $\delta$. The operator (\ref{612}) is equivalent to the map $f\mapsto {\partial f\over \partial \ov{z}}$ so it is essentially the same as the case of (\ref{611}) to deduce that it has dense range.
\end{proof}

\begin{rem}
In particular, the transversality of the linearization $D_{A, u}$ implies that the moduli space $N_d$ is a correct one to define the quantum Kirwan map. The remaining is to show that the evaluation maps indeed give a pseudo-cycle. This can be seen by giving suitable compactifications of the moduli space on which the evaluation maps extend continuously. In the last two sections, we will first give the stable map compactification of the lowest nontrivial moduli, and then construct a compactification (which we call Uhlenbeck compactification) for arbitrary degrees which allows us to extend the evaluation maps.
\end{rem}

\subsection{Stable map compactification of ${\mc M}_{1, 1}^{\mb A}\left( {\mb C}^N, 1 \right)$}

In this subsection we identify the compactification $\ov{\mc M}^{\mb A}_{1,1}\left( {\mb C}^N, 1 \right)$ defined in Subsection \ref{subsection42}, which is the smallest nontrivial case. Depending on the type of degeneration, $\ov{\mc M}_{1, 1}^{\mb A}\left({\mb C}^N, 1 \right)$ is stratified as described in Figure \ref{figure1}: the open stratum is ${\mc M}_{1, 1}^{\mb A}\left({\mb C}^N, 1 \right)$; two lower strata of depth 1 are: 1) the stratum ${\mc T}_1$ corresponding to the moduli of degree 1 holomorphic spheres in ${\mb P}^{N-1}$ with one marked point and a ghost vortex attached to the sphere, which is the same as ${\mc M}_2\left( {\mb P}^{N-1}, 1 \right)$ and 2) the stratum ${\mc T}_2$ corresponding to the moduli of $(0, 1)$-marked degree 1 affine vortices with two ghost components attached at $\infty$; the lowest stratum ${\mc S}$ is the moduli of holomorphic spheres with two ghost components attached at $\infty$. Note that, points in each stratum only depend on the isomorphism class of their nontrivial components and the other ghost components are attached to the nontrivial components in a unique way.

\begin{figure}[htbp]
\centering
\includegraphics[scale=0.60]{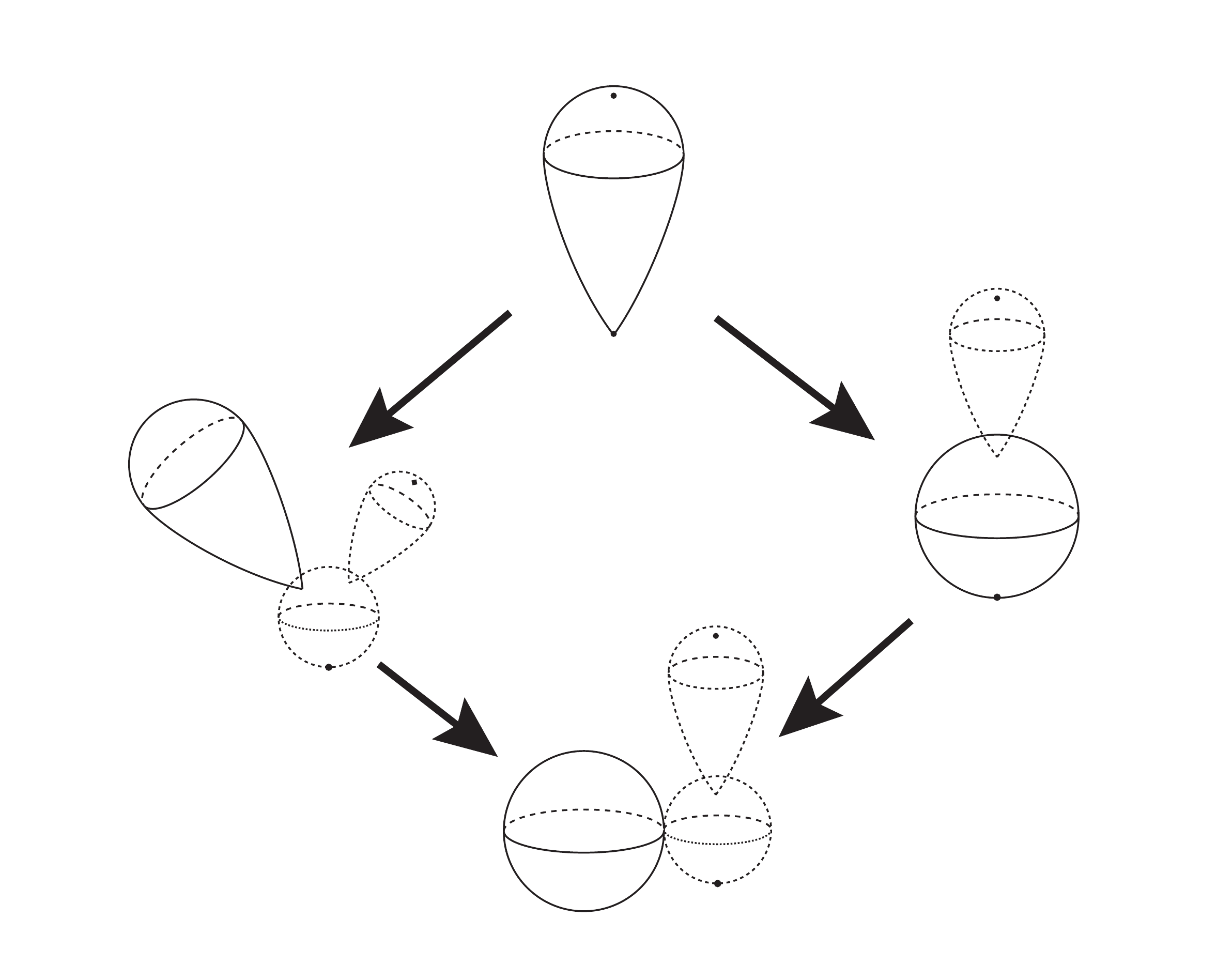}
\caption{The stratification of $\ov{\mc M}_{1, 1}^{\mb A} \left( {\mb C}^N, 1\right)$.}
\label{figure1}
\end{figure}

We need to find the correct compactification of $N_1$. Note that $N_1 = {\mc O}(1)^{\oplus N}$ where ${\mc O}(1) \to {\mb P}^{N-1}$ is the degree 1 line bundle. For ${\bf a}:= (a_1, \ldots, a_N) \neq 0$, we use the homogeneous coordinates $[a_1, \ldots, a_N]$ to denote a point in the base ${\mb P}^{N-1}$, ${\bf a}$ to denote a basis vector of the fibre of ${\mc O}(-1)$ over the point $[a_1, \ldots, a_N]$, and ${\bf a}^*$ the dual basis of the fibre of ${\mc O}(1)$ over $[a_1, \ldots, a_N]$. Then a point in $N_1$	will be denoted by $(b_1, \ldots, b_N) {\bf a}^*$. 

There is a natural compactification of ${\mc O}(1)^{\oplus N}$ by the projective bundle
\begin{align}
{\mb P}({\mc E}) := {\mb P}\left( {\mc O}(1)^{\oplus N} \oplus {\mb C} \right)
\end{align}
where ${\mc O}(1)^{\oplus N}$ embeds into ${\mb P}({\mc E})$ as
\begin{align}\label{59}	
\left( b_1 {\bf a}^*, \ldots, b_N {\bf a}^* \right) \mapsto \left[ b_1 {\bf a}^*, \ldots, b_N {\bf a}^*, 1 \right].
\end{align}

We see that ${\mb P}({\mc E}) \setminus N_1 = {\mb P}^{N-1} \times {\mb P}^{N-1}$. Remember that, the moduli space of genus zero, degree 1, 2-marked stable map to ${\mb P}^{N-1}$ is isomorphic to the blow up $Bl_\Delta\left( {\mb P}^{N-1}\times {\mb P}^{N-1}\right)$ along the diagonal $\Delta \subset {\mb P}^{N-1}\times {\mb P}^{N-1}$. Regard $\Delta \subset {\mb P}({\mc E})$ as a codimension $N$ subvariety. We denote by 
\begin{align}
\ov{N}_1:= Bl_\Delta{\mb P}({\mc E})
\end{align}
the blown-up. Denote the exceptional divisor by $E\simeq {\mb P}\left(N_\Delta \oplus {\mb C} \right)$, where $N_\Delta$ is the normal bundle of $\Delta$ in ${\mb P}^{N-1} \times {\mb P}^{N-1}$. Now the space $\ov{N}_1$ can be stratified: the top stratum is just $N_1$; there are two strata of depth one, which are $T_1:= {\mb P}^{N-1} \times {\mb P}^{N-1} \setminus \Delta$ and $T_2:= E\setminus {\mb P}(N_\Delta)$; there is a single lowest stratum of depth two, which is $S := {\mb P}\left( N_\Delta \right)$. For each point ${\bf x}= \left( [a_1, \ldots, a_N], [a_1, \ldots, a_N] \right) \in \Delta$, the fibre of $T_2$ over ${\bf x}$ has the coordinates $\left[  \left( v_1 {\bf a}^*, \ldots, v_N {\bf a}^*\right), \left( - v_1 {\bf a}^*, \ldots, - v_N {\bf a}^* \right), 1                     \right]$, where $(v_1, \ldots, v_N) \bot (a_1, \ldots, a_N)$ in ${\mb C}^N$ and hence
\begin{align}
 \left( v_1 {\bf a}^*, \ldots, v_N {\bf a}^*\right)\in {\rm Hom} \left( {\mc O}(1)|_{[{\bf a}]}, {\bf a}^{\bot} \right) = T_{[{\bf a}]} {\mb P}^{N-1}.
\end{align}

Now we prove that the map $\Phi_{HK}^1$ defined in Corollary \ref{cor63} extends to 
\begin{align}
\ov{\Phi}_{HK}^1: \ov{N}_1 \to \ov{\mc M}^{\mb A}_{1, 1}\left( {\mb C}^N, 1 \right)
\end{align}
which is a homeomorphism with respect to the Gromov convergence defined in Subsection \ref{subsection43} and respects the stratifications of the domain and the target. 

We first define the extension as follows.
\begin{enumerate}
\item $\ov{\Phi}_{HK}^1: T_1 \to {\mc T}_1$ assigns to $[a_1, \ldots, a_N, b_1, \ldots, b_N]$ the $(1, 1)$-marked stable affine vortex in ${\mc T}_1$ whose nontrivial component is the equivalent to the holomorphic sphere $z \mapsto [ a_1 z + b_1, \ldots, a_N z + b_N]$.

\item $\ov{\Phi}_{HK}^1$ restricted to $T_2= E\setminus {\mb P}(N_\Delta)$ assigns each $[(v_1, \ldots, v_N) {\bf a}^*, -(v_1, \ldots, v_N) {\bf a}^*, 1]$ to the equivalence class of $(1, 1)$-marked stable affine vortices in ${\mc T}_2$ whose nontrivial component is equivalent to the vortex $\wt{\Phi}_{HK}^1\left([a_1, \ldots, a_N, - v_1, \ldots, - v_N ] \right)$.
\end{enumerate}

Now we have to prove
\begin{lemma} 
$\ov{\Phi}_{HK}^1$ is continuous.
\end{lemma}

\begin{proof} 
It is continuous on each stratum. Hence we take a sequence ${\bf x}_i \in N_1$ which converges to a point ${\bf x}_\infty \in \ov{N}_1 \setminus N_1$. We may write $ {\bf x}_i= \left[ \left( b_{1, i}, \ldots, b_{N, i}\right) {\bf a}_i^*, w_i\right]\in {\mb P}\left( {\mc O}(1)^{\oplus N} \oplus {\mb C}\right)$ with ${\bf a}_i$, ${\bf b}_i$ unit vectors, $w_i \in {\mb C}^*$, and $\lim_{i \to \infty} {\bf a}_i = {\bf a}, \lim_{i \to \infty} {\bf b}_i = {\bf b}$, $\lim_{i \to \infty} w_i = 0$. Let $d_i:= d_{{\mb P}^{N-1}}([ {\bf a}_i], [{\bf b}_i])$.

We first observe that
\begin{enumerate}
\item ${\bf x}_\infty \in T_1 \Longrightarrow [{\bf a}] \neq [{\bf b}]$ in ${\mb P}^{N-1}$;

\item ${\bf x}_\infty \in T_2 \Longrightarrow [{\bf a}] = [{\bf b}]$ and $\lim_{i \to \infty} |w_i|^{-1} d_i < \infty$;

\item ${\bf x}_\infty \in S \Longrightarrow [{\bf a}] = [{\bf b}]$ and $\lim_{i \to \infty} | w_i |^{-1} d_i = \infty$.
\end{enumerate}

\begin{enumerate}
\item In the first and the third case, we have
\begin{align}
\lim_{i \to \infty} \inf_{z \in {\mb C}}  \left| {\bf a}_i z + w_i^{-1} {\bf b}_i \right|^2 \geq \lim_{i \to \infty} |w_i|^{-2} d_i^2 = +\infty.
\end{align}
Then look at the Kazdan-Warner equation
\begin{align}
\Delta_0 h_i + {1\over 2} \left( e^{-2h_i}\sum_{j=1}^N \left| a_{j, i} z + \epsilon_i^{-1} b_{j, i} \right|^2 -1 \right) = 0.
\end{align}
We denote
\begin{align}
h_i'(z):= h_i(z) - {1\over 2} \log \left( \left| \ora{a}_i z +  \epsilon_i^{-1} \ora{b}_i \right|^2\right).
\end{align}
Then
\begin{align}
\Delta_0 h_i' + {1\over 2} ( e^{- 2 h_i'} -1 ) = - {1\over 2} \Delta_0 \log \left( \left| \ora{a}_i z + \epsilon_i^{-1} \ora{b}_i \right|^2\right).
\end{align}
Note that the right hand side converges to zero uniformly over ${\mb C}$, which implies that $h_i'$ converges to zero. 

\item In the first case, take the M\"obius transformation
\begin{align}
\begin{array}{cccc}
g_i: & S^2 & \to &  S^2\\
      & w & \mapsto & z= w_i^{-1} w.
\end{array}
\end{align}
We see the sequence of maps
\begin{align}
\pi_\nu\left( g_i^*( e^{-h_i} \phi_{j, i} )\right)(w) = e^{-h_i(\epsilon_i^{-1} w)} w_i^{-1}( a_{j, i} w + b_{j, i}).
\end{align}
converges on ${\mb C}\subset S^2$ to 
\begin{align}
w\mapsto \left[ a_1 w + b_1,  \ldots, a_N w + b_N          \right]\in {\mb P}^{N-1}
\end{align}
which extends to a nontrivial holomorphic map from ${\mb P}^1$ to ${\mb P}^{N-1}$ of degree 1. By Definition \ref{defn38}, this means that in the first case, the sequence $\Phi_{HK}^1({\bf x}_i)\in {\mc M}_{1,1}^{\mb A}\left( {\mb C}^N, 1 \right)$ converges to $\ov{\Phi}_{HK}^1({\bf x}_\infty)$. 

\item In the third case, for large $i$ we write ${\bf b}_i = c_i {\bf a}_i + {\bf y}_i$ with $c_i \in {\mb C}$ and ${\bf y}_i \bot {\bf a}_i$. Then by the condition ${\bf x}_\infty \in S$, we have $\lim_{i \to \infty} d_{{\mb P}^{N-1}}([{\bf a}_i], [{\bf b}_i])^{-1} {\bf y}_i$ exists. Then take the sequence of M\"obius transformations 
\begin{align}
z= g_i(w) = {d_i \over w_i} w - c_i w_i^{-1}.
\end{align}
We see
\begin{align}
\lim_{i \to \infty} \left[ g_i^* \phi_{1, i}, \ldots, g_i^* \phi_{N, i}    \right] = \lim_{i \to \infty} \left[ a_{1, i} d_i w + y_{1, i}, \ldots, a_{N, i} d_i w + y_{N, i}          \right] = [ a_1 w + y_1, \ldots, a_N w + y_N]
\end{align}
which is a degree one holomorphic sphere in ${\mb P}^{N-1}$. Adding proper ghost components, we see that this means the sequence $\Phi_{HK}^1({\bf x}_i)$ converges to $\ov{\Phi}_{HK}^1({\bf x}_\infty)$.

\item In the second case, write ${\bf b}_i = c_i {\bf a}_i + {\bf y}_i$ with $c_i \in {\mb C}$, ${\bf y}_i \bot {\bf a}_i$. Since ${\bf x}_i = \left[ ( b_{1,i}, \ldots, b_{N, i}) {\bf a}_i^*, w_i \right] \in {\mb P}\left( {\mc O}(1)^{\oplus N}\oplus {\mb C}\right)$ converge to ${\bf x}_\infty \in {\mb P}\left( N_\Delta \oplus {\mb C}\right) \setminus {\mb P}\left(N_\Delta \right)$, the limit 
$$ {\bf v}:= \lim_{i \to \infty} w_i^{-1} {\bf y}_i \in {\bf a}^{\bot}\subset {\mb C}^N$$
exists. Then consider the sequence of translations
\begin{align}
z= t_i(w) = w - c_i w_i^{-1}.
\end{align}
Then we see the sequence of polynomials
\begin{align}
t_i^* \phi_{j, i} = a_{j, i} ( w- c_i w_i^{-1}) + w_i^{-1} b_{j, i} = a_{j, i} w- w_i^{-1} y_{j, i}
\end{align}
converge to $a_j w - v_j$. By the continuous dependence of the solution to the Kazdan-Warner equation on the given $N$ polynomials, we see that $t_i^* \wt{\Phi}_{HK}^1( {\bf x}_i)$ converge uniformly on any compact subset to the affine vortex $\wt{\Phi}_{HK}^1( [ a_1, \ldots, a_N, v_1, \ldots, v_N])$. Hence $\Phi_{HK}^1({\bf x}_i)$ converge to $\ov{\Phi}_{HK}^1({\bf x}_\infty)$.
\end{enumerate}
\end{proof}

\subsection{The Uhlenbeck compactification of ${\mc M}_{1, 1}^{\mb A} \left({\mb C}^N, d \right)$ and the quantum Kirwan map}

We define the Uhlenbeck compactification to be a quotient space of the stable map compactification, by only remembering the sum of the degrees of the components of the stable map which doesn't contain the marked point $0$.

\begin{prop}
The Uhlenbeck compactification $\ov{\mc M}^{{\mb A},U}_{1, 1}\left( {\mb C}^N, d \right)$ is homeomorphic to ${\mb P}^{N(d+1)-1}$.
\end{prop}

\begin{proof} We see that we have a filtration ${\mb P}^{N(d+1)-1} \supset {\mb P}^{Nd -1} \supset \cdots \supset {\mb P}^{N-1}$ where the inclusion ${\mb P}^{Nk-1} \to {\mb P}^{N(k+1)-1}$ is given by 
\begin{align}
[ a_{Nk}, a_{Nk-1}, \ldots, a_1 ] \mapsto [0, \ldots, 0, a_{Nk}, a_{Nk-1}, \ldots, a_1].
\end{align}
So $N_d = {\mb P}^{N(d+1)-1} \setminus {\mb P}^{Nd-1}$ and ${\mb P}^{N(d+1)-1} = \cup_{0\leq k \leq d} N_k$. We define the extension of $\Phi_{HK}^d$
\begin{align}
\ov{\Phi}_{HK}^{d, U}: {\mb P}^{N(d+1)-1} \to \ov{\mc M}^{{\mb A}, U}_{1, 1}\left( {\mb C}^N, d \right)
\end{align}
to be the map such that for ${\bf x}\in N_k \subset {\mb P}^{Nd-1}$, $\ov\Phi^{d, U}_{HK}({\bf x})$ is the equivalence class of stable affine vortices whose primary component is equivalent to $\wt\Phi^k_{HK}({\bf x})$. It remains to show that this map is continuous with respect to the degeneration of affine vortices.

Indeed, suppose ${\bf x}_i \in N_d$ and $\lim_{i \to \infty} {\bf x}_i = {\bf x}_\infty \in N_k$ for $k < d$. Represent ${\bf x}_\infty$ by $N$ polynomials $\left\{ \psi_{j, \infty} \right\}_{1\leq j \leq N}$ with maximal degree $k$. Without loss of generality, we can assume that ${\bf x}_i$ can be represented by $N$ polynomials $\left\{ \wt{\psi}_{j, i} \right\}_{1\leq j \leq N}$ such that ${\rm deg} \wt{\psi}_{j, i}$ is independent of $i$ and $d_j:= {\rm deg} \wt{\psi}_{j, i} \geq {\rm deg} \psi_{j, \infty} =: d_{j, \infty}$. Then $d_j > d_{j, \infty}$ implies that $d_j - d_{j, \infty}$ zeroes of $\wt{\psi}_{j, i}$ diverge to infinity. Hence we can write
\begin{align}\label{eqn735}
\wt{\psi}_{j, i}(z) =  \psi_{j, i}(z) \prod_{s=1}^{d_j - d_{j, \infty}} \left( 1 - {z\over w_{j, i, s}} \right)
\end{align}
with $\lim_{i \to \infty} \psi_{j, i}= \psi_{j, \infty}$, $\lim_{i \to \infty} |w_{j, i, s}| = \infty$. Then for each $i$ there exists functions $h_i$ solving the Kazdan-Warner equation
\begin{align}
 \Delta h_i + {1\over 2} \left(  e^{- 2 h_i} \sum_{j=1}^N \left| \wt{\psi}_{j, i} \right|^2 -1 \right) = 0.
\end{align}
By the compactness theorem of Ziltener (Theorem \ref{thm49}), a subsequence of $\left( d - \partial h_i + \ov\partial h_i, e^{- h_i} \left( \wt{\psi}_{1, i}, \ldots, \wt{\psi}_{N, i} \right) \right)$ converges to a $(1, 1)$-marked stable affine vortex $\wt{\bf W}^\infty$; in particular, a subsequence converges uniformly on any compact subset of ${\mb C}$ to the primary component of $\wt{\bf W}^\infty$. This implies that for each $j$, $\left\{ e^{-h_i} \wt{\psi}_{j, i} \right\}_i$ has a convergent subsequence. Since $\lim_{i \to \infty} \wt{\psi}_{j, i}= \psi_{j, \infty}$, this implies that $h_i$ converges on ${\mb C}$ to a smooth function $h_\infty$, which solves the equation
\begin{align}
 \Delta h_\infty + {1\over 2} \left( e^{- 2 h_\infty} \sum_{j=1}^N \left| \psi_{j, \infty} \right|^2 -1 \right)  = 0.
\end{align}
This implies that the primary component of $\wt{\bf W}^\infty$ is equivalent to $\left( \wt{\Phi}_{HK}^k({\bf x}_\infty), 0 \right)$. Hence we have proved that any subsequence of $\Phi_{HK}^d ({\bf x}_i)$ has a subsequence converging to $\ov{\Phi}_{HK}^{d, U}({\bf x}_\infty)$. This implies the continuity of $\ov{\Phi}_{HK}^{d, U}$.
\end{proof}

\begin{prop}
The line bundle associated to the Poincar\'e bundle $\ov{\mc P}_0^U \to \ov{\mc M}^U_{1,1}\left( {\mb C}^N, d \right)$ is isomorphic to ${\mc O}(1) \to {\mb P}^{Nd-1}$.
\end{prop}
\begin{proof}
It suffices to check on each $N_k$.
\end{proof}

Now the evaluation $ev_\infty$ doesn't extends to $\ov{\mc M}_{1,1}^U\left( {\mb C}^N, d \right)$, but it doesn't affect our computation of the Kirwan map. Indeed we can blow up ${\mb P}^{N(d+1)-1}$ along ${\mb P}^{Nd-1}= \cup_{0\leq k \leq d} N_k$ on which $ev_\infty$ extends continuously. The blown-up is denoted by $\ov{N}_d^*$.

Now we compute the quantum Kirwan map, the result of which is of no surprise. $H_{U(1)}^*\left( {\mb C}^N \right)$ is generated by the universal first Chern class $u$ of degree 2. For any $m\geq 0$, write $m = d_m N + r$ with $0 \leq r \leq N-1$. Denote by $c\in H^2 \left( {\mb P}^{N-1}\right)$ the generator. Then by the definition (\ref{eqn57})
\begin{multline}
\kappa_Q( u^m) = \sum_{d\geq 0} \sum_{0 \leq i \leq N-1} \left\langle (ev_0)^*( u^m)  \cup (ev_\infty)^*( c^i), \left[\ov{N}_d^*\right]  \right\rangle \cdot c^{N-1-i} \otimes q^d\\
= \left\langle (ev_0)^*( u^m)  \cup (ev_\infty)^*( c^{N-1-r}), \left[\ov{N}_{d_m}^*\right]  \right\rangle \cdot c^r \otimes q^{d_m}= c^r \otimes q^{d_m} = \kappa_Q(u) *_q \kappa_Q(u) *_q \cdots *_q \kappa_Q(u).
\end{multline}
Hence $\kappa_Q: H_{U(1)}^*({\mb C}^N) \to QH^*( {\mb P}^{N-1}, \Lambda)$ is a ring homomorphism. It extends to a homomorphism 
\begin{align}
\kappa_Q^\Lambda: H_{U(1)}^*({\mb C}^N, \Lambda) \to QH^* ( {\mb P}^{N-1} , \Lambda)
\end{align}
linearly over $\Lambda$, with kernel generated by $q- u^N$.

\bibliography{symplectic_ref}

\providecommand{\bysame}{\leavevmode\hbox to3em{\hrulefill}\thinspace}
\providecommand{\MR}{\relax\ifhmode\unskip\space\fi MR }
\providecommand{\MRhref}[2]{%
  \href{http://www.ams.org/mathscinet-getitem?mr=#1}{#2}
}
\providecommand{\href}[2]{#2}
\begin{thebibliography}{10}

\bibitem{Bradlow_stable_pairs}
Steven Bradlow, \emph{Special metrics and stability for holomorphic bundles
  with global sections}, Journal of Differential Geometry \textbf{33} (1991),
  169--214.

\bibitem{Gaio_Salamon_2005}
Ana Gaio and Dietmar Salamon, \emph{Gromov-{W}itten invariants of symplectic
  quotients and adiabatic limits}, Journal of symplectic geometry \textbf{3}
  (2005), no.~1, 55--159.

\bibitem{Jaffe_Taubes}
Arthur Jaffe and Clifford Taubes, \emph{Vortices and monopoles}, Progress in
  physics, no.~2, Birkh\"auser, 1980.

\bibitem{McDuff_Salamon_2004}
Dusa McDuff and Dietmar Salamon, \emph{${J}$-holomorphic curves and symplectic
  topology}, Colloquium publications, vol.~52, American mathematical society,
  2004.

\bibitem{Mundet_Hitchin_Kobayashi}
Ignasi {Mundet i Riera}, \emph{A {H}itchin-{K}obayashi correspondence for
  {K}\"ahler fibrations}, Journal f\"ur die Reine und Angewandte Mathematik
  \textbf{528} (2000), 41--80.

\bibitem{Taubes_vortex}
Clifford Taubes, \emph{Arbitrary ${N}$-vortex solutions to the first order
  {G}inzburg-{L}andau equations}, Communications in {M}athematical {P}hysics
  \textbf{72} (1980), no.~3, 277--292.

\bibitem{VW_Classification}
Sushimita Venugopalan and Chris Woodward, \emph{Classification of vortices}, In
  preparation, 2012.

\bibitem{chris_quantum_kirwan}
Chris Woodward, \emph{Quantum {K}irwan morphism and {G}romov-{W}itten
  invariants of quotients}, arXiv:1204.1765, April 2012.

\bibitem{Ziltener_thesis}
Fabian Ziltener, \emph{Symplectic vortices on the complex plane and quantum
  cohomology}, Ph.D. thesis, Swiss {F}ederal {I}nstitute of {T}echnology
  {Z}urich, 2005.

\bibitem{Ziltener_Deday}
\bysame, \emph{The invariant symplectic action and decay for vortices}, Journal
  of {S}ymplectic {G}eometry \textbf{7} (2009), no.~3, 357--376.

\bibitem{quantumkirwan}
\bysame, \emph{A quantum {K}irwan map: bubbling and {F}redholm theory}, Memiors
  of the {A}merican Mathematical Society (2012).

\end{thebibliography}
	
\bibliographystyle{amsplain}
\end{document}